\documentclass[a4paper, 12pt]{amsart}
\usepackage{amsmath, amsthm, amscd, amssymb, amsfonts, amsxtra, amssymb, latexsym}
\usepackage{verbatim}
\usepackage{graphicx} 
\usepackage{float}
\usepackage{enumerate}
\usepackage{mathrsfs}
\usepackage[mathscr]{euscript}
\usepackage{etoolbox} 

\usepackage{tikz}
\usetikzlibrary{calc}
\usetikzlibrary{arrows.meta}

\usepackage{booktabs} 
\usepackage{makecell} 

\usepackage{hyperref}
\hypersetup{colorlinks = true,	allcolors  = blue}

\usepackage{xcolor}

\def\black{\color{black}}

\newcommand{\sk}{\smallskip}
\newcommand{\msk}{\medskip}

\newcommand{\N}{\mathbb{N}}
\newcommand{\Z}{\mathbb{Z}}
\newcommand{\Q}{\mathbb{Q}}

\newcommand{\C}{\mathbb{C}}
\newcommand{\ff}{\mathbb{F}}

\newcommand{\G}{\Gamma}

\numberwithin{equation}{section}

\newtheorem{thm}{Theorem}[section]
\newtheorem{prop}[thm]{Proposition}
\newtheorem{lem}[thm]{Lemma}
\newtheorem{coro}[thm]{Corollary}

\theoremstyle{definition}
\newtheorem{rem}[thm]{Remark}
\newtheorem{exam}[thm]{Example}
\newtheorem{defi}[thm]{Definition}

\newtheorem*{openquest}{Open question}

\begin{document} \sloppy
\title[{Isospectral Cayley graphs with even and odd spectrum}]
{Isospectral Cayley graphs with \\ even and odd spectrum}
	
\author{Paula M.\@ Chiapparoli, Ricardo A.\@ Podest\'a}
	
\dedicatory{\today}
	
	\thanks{Partially supported by CONICET, FonCyT and SECyT-UNC}
	\address{Ricardo A.\@ Podest\'a. FaMAF--CIEM (CONICET), Universidad Nacional de C\'ordoba. 
	Av.\@ Medina Allende 2144, (5000), C\'ordoba, Argentina. \newline
		{\it E-mail: podesta@famaf.unc.edu.ar}}
	\address{Paula M.\@ Chiapparoli. FaMAF--CIEM (CONICET), Universidad Nacional de C\'ordoba. 
	Av.\@ Medina Allende 2144, (5000), C\'ordoba, Argentina.
		\newline {\it E-mail: paula.chiapparoli@mi.unc.edu.ar}}

\begin{abstract}
For a group $G$ and subsets $S,T \subset G$ we introduce the \textit{mirror di-Cayley} graph $MX(G;S,T)$ and \textit{mirror di-Cayley sum} graph $MX^+(G;S,T)$ with connection sets $S$ and $T$. We refer to them as MDCGs for short and denote them indistinctly by $MX^*(G;S,T)$. 
We then consider the family $\mathcal{F}$ of those MDCGs with $T \in \mathcal{S}$, where $\mathcal{S}= \big\{ \{e\}, S, S \cup \{e\} \big\}$. 

We compute the spectra of the graphs $MX^*(G;S,T)$, with $T \in \mathcal{S}$, in terms of those of the corresponding Cayley graphs $X^*(G,S)$. 
We show that if $X(G,S)$ has integral spectrum then $MX^*(G;S,T)$ is also integral for any $T \in \mathcal{S}$, but $MX^*(G;S,S)$ has even spectrum (all even eigenvalues) and $MX^*(G;S,S \cup \{e\})$ has odd spectrum (all odd eigenvalues), an interesting phenomenon which seems to be new. We then study isospectrality between different pairs of MDCGs in terms of the isospectrality of the underlying Cayley graphs.

Finally, using unitary Cayley graphs $X(R,R^*)$ over a finite commutative ring $R$, which is known to be integral, we construct pairs of integral isospectral mirror di-Cayley (sum) graphs $\{ MX(R;R^*, T), MX^+(R;R^*, T) \}$, both with even (resp.\@ odd) spectrum for $T=R^*$ (resp.\@ $T=R^* \cup \{0\}$). All these examples can be seen as Cayley (sum) graphs over $G=R \times \Z_2$, hence obtaining pairs of even and odd isospectral Cayley graphs of the form $\{\G, \G^+\}$. 
\end{abstract}

\maketitle

\section{Introduction} \label{sec: intro} 
Among several relevant problems in spectral theory, there are two which are of special interest to us:
the integrality of the spectrum and the isospectral problem.
A lot of research has been done on each of them, both in general and for Cayley graphs. For these problems, we will restrict ourselves to Cayley graphs. 

The first problem originated with the question \textit{“Which graphs have integral spectra?”} by Harary and Schwenk in 1973 (\cite{HS}). The problem seems to be very hard in
great generality. 
More than three decades later, Abdollahi and Vatandoost asked \textit{“Which Cayley graphs are integral?”} in \cite{AV}. There are some good results when one restrict to this family of graphs (see for instance Section 3 in \cite{LZ}). 

The second problem is also classic and asks which non-isomorphic graphs are isospectral. A way to study this problem is by giving pairs of isospectral graphs with different properties. 
A good source for spectral properties of Cayley graphs is the survey of Liu and Zhou (\cite{LZ}), where you can find many results and references both of integral Cayley graphs and of isospectral Cayley graphs (see Sections 3 and 4, respectively).

In this work we consider both problems together.
\textit{Is it possible to obtain pairs of integral isospectral Cayley graphs?} If so, 
\textit{can we do it only with even or odd eigenvalues?}
In \cite{PV}, it is proved that, for $G$ abelian and $S$ a symmetric subset of $G$ not containing $0$, the Cayley graph $X(G,S)$ and the Cayley sum graph $X^+(G,S)$ are non-isospectral equienergetic graphs. There, many other interesting results are obtained by considering the special case of unitary Cayley graphs over rings, i.e.\@ take as $G$ a finite commutative ring with identity $R$ and as $S$ the group of units $R^*$. 

Motivated by the previous question, here we construct isospectral pairs of integral Cayley graphs with the extra property of having either all even or all odd eigenvalues, a phenomenon that seems to be new. 
We will do this by introducing three families of mirror di-Cayley (sum) graphs. 
The final result will be obtained using the unitary Cayley (sum) graphs $X(R,R^*)$ and $X^+(R, R^*)$, for $R$ a ring as above.

\subsubsection*{Mirror di-Cayley graphs}   
Let $G$ be a group and $S$ and $T$ be two subsets of $G$. 
We define the \textit{mirror di-Cayley graph} 
$$MX(G;S,T)$$ as the graph having vertex set $G \times \Z_2$ and where the directed edges are the edges of the Cayley graph $X(G,S)$ replicated in 
$$ X(G,S) \times \{0\} \qquad \text{and} \qquad X(G,S) \times \{1\}$$ (called the \textit{mirrors}) together with the crossing edges, that is the edges between the mirrors (see Definition \ref{def: dicayleys}). 
Similarly, we define the sum version of this graph, the  \textit{mirror di-Cayley sum graph} $MX^+(G;S,T)$, using the Cayley sum graph $X^+(G,S)$ as mirrors.
Here, and throughout the paper, we will denote by 
$X^*(G,S)$ the Cayley graph $X(G,S)=Cay(G,S)$ and the Cayley sum graph $X^+(G,S)=Cay^+(G,S)$ considered together. Similarly, we will use the notation 
     $MX^*(G;S,T)$ 
when considering the mirror di-Cayley graph $MX(G;S,T)$ and the mirror di-Cayley sum graph $MX^+(G;S,T)$ together. We refer to them as MDCGs for short.

Here, we will consider a family of three particular kind of MDCGs, namely those $MX^*(G;S,T)$ having $T=\{e\}, S$ or $S\cup \{e\}$, and focus on their spectral properties. 
That is, we will study the family of graphs 
    $$ \mathcal{F} = \big\{ MX^*(G;S,T) : \text{$G$ a group}, S\subset G \big\}_{T \in \mathcal{S}}, $$ 
where $\mathcal{S} = \{ \{e\}, S, S\cup \{e\} \}$.
Other choices of $T$ are also of interest. However, here we restrict ourselves to these subsets since, as we will see, the associated graphs have nice decompositions as products of graphs.

\subsubsection*{Graph spectral definitions}
Let $\G=(V,E)$ be a graph with $n$ vertices, where $V$ denotes the vertex set and $E$ the edge set. The eigenvalues of $\Gamma$ are the eigenvalues 
$\{\lambda_i\}_{i=1}^{n}$ of its adjacency matrix $A$. We denote this set by $Eig(\G)$. The \textit{spectrum} of $\Gamma$, denoted
	$$ Spec(\Gamma)=\{[\lambda_{i_1}]^{m_{i_1}},\cdots,[\lambda_{i_s}]^{m_{i_s}}\} $$
is the (multi)set of all the different eigenvalues $\{\lambda_{i_j}\}$ of $\Gamma$, counted with their multiplicities $\{m_{i_j}\}$, that is $m(\lambda_{i_j})=m_{i_j}$. 

The spectrum of $\Gamma$ is called \textit{symmetric} if $Spec(\G)=-Spec(\G)$, that is the multiplicities of $\lambda$ and $-\lambda$ coincide for any $\lambda$, in symbols 
    $$m(\lambda) = m(-\lambda)$$ 
for every $\lambda\in Spec(\Gamma)$. 
The spectrum of $\G$ is \textit{almost symmetric} if  $m(\lambda) = m(-\lambda)$ for every eigenvalue $\lambda$ with $\lambda \ne  \lambda_1$, where $\lambda_1$ is the principal eigenvalue (see Definition~2.13 in \cite{PV}).

The spectrum of $\G$ is said to be \textit{real} or \textit{integral} if 
$$ Spec(\Gamma)\subset\mathbb{R} \qquad \text{or} \qquad Spec(\Gamma)\subset\mathbb{Z},$$ 
respectively. 
We now introduce the following notion for integral graphs. We say that $\G$ is \textit{even} (resp.\@ \textit{odd}) or that it \textit{has even} (resp.\@ \textit{odd}) \textit{spectrum} if it is integral and all its eigenvalues are even (resp.\@ odd). That is 
$$ \G \text{ is even \, (resp.\@ odd) } \qquad \Leftrightarrow \qquad Spec(\G) \subset 2\Z \text{ \, (resp.\@ $2\Z+1$)}.$$
	
We recall that if $\Gamma$ is a $k$-regular graph, then $\lambda_1 = k$. Moreover, the multiplicity of $\lambda_1$ equals the number of connected components of $G$ and, therefore, $\Gamma$ is connected if and only if $m(\lambda_1) = 1$. Spectrally, a $k$-regular graph $G$ is bipartite if and only if the spectrum of $G$ is symmetric, which happens if and only if $-k$ is an eigenvalue of $\Gamma$.

Let $\Gamma_1$ and $\Gamma_2$ be two graphs with the same number of vertices. The graphs $\Gamma_1$ and $\Gamma_2$ are said to be \textit{isospectral} if 
    $$ Spec(\Gamma_1) = Spec(\Gamma_2).$$ 
	
In this paper we will study the spectrum of MDCGs in the family $\mathcal{F}$ and isospectrality between non-isomorphic graphs in the family.

\subsubsection*{Outline and results}
In Section \ref{sec: 3 families} we introduce the family of mirror di-Cayley (sum) graphs $MX^*(G;S,T)$ with $T=\{e\}, S$ or $S\cup \{e\}$ and give their basic properties (see Proposition \ref{prop: directedness}). 
In Proposition \ref{prop: cayley structure} we show that the graph 
$MX^*(G;S,T)$ is a Cayley graph over $G\times \Z_2$, namely $$MX^*(G;S,T) = X^*(G \times \Z_2, (S \times \{0\}) \cup (T\times \{1\})).$$
In Theorem \ref{teo: prods} we show that the graphs 
    $$ \G_{G;S,T}^* := MX^*(G;S,T) $$ 
are products of the underlying Cayley graph $\G_{G,S}^*=X^*(G,S)$ and the 2-path $P_2$ or the 2-path with loops $\mathring{P}_2$; namely, 
    $$	\G_{G;S,\{e\}}^*  = \G_{G,S}^* \Box P_2, \qquad 
	\G_{G;S,S}^*= \G_{G,S}^* \times \mathring{P_2}, \qquad \G_{G;S,S\cup\{e\}}^*  = \G_{G,S}^* \boxtimes P_2, $$
where $\Box$, $\times$ and $\boxtimes$ denote the Cartesian, direct and  strong products, respectively. 
	
In Section \ref{sec: spec}, we study the spectrum of the mirror di-Cayley (sum) graphs $\G_{G;S,\{e\}}^*$, $\G_{G;S,S}^*$ and $\G_{G;S,S\cup \{e\}}^*$
for arbitrary pairs $G$ and $S$, that is, we study the set 
    $$ \big\{ Spec( \G_{G;S,T}^*) : \text{ $G$ a group, $S\subset G$} \big\}_{T \in \mathcal{S}}.$$
Using the product decomposition of the MDCGs, in Proposition \ref{prop: spec bicayleys} we get the spectrum of each mirror di-Cayley (sum) graph $MX^*(G;S,T)$ in terms of the spectrum of the Cayley (sum) graph $X^*(G,S)$. Moreover, using the known description of the spectrum of Cayley graphs $X(G,S)$ and Cayley sum graphs $X^+(G,S)$ in terms of the irreducible characters $\chi$ of $G$, in Theorems \ref{thm: spec 3fam chars} and \ref{thm: spec 3fam chars sum} we give the spectrum of the MDCGs in terms of the numbers $\chi(S) =\sum_{s \in S} \chi(s)$, for $\chi$ running in the set of irreducible characters of $G$.

In Sections \ref{sec: integral} and \ref{sec: integral GR} we study integral graphs having even and odd spectrum. 
In Section \ref{sec: integral} we introduce the notion of even and odd spectrum (see Definition \ref{defi: even/odd}). In Corollary \ref{coro: integral} we show that $MX^*(G;S,T) \in  \mathcal{F}$ is integral if and only if $X^*(G,S)$ is integral and, in this case, that 
    $$ MX^*(G;S,S) \text{ is even  \qquad and \qquad} MX^*(G;S,S \cup \{e\}) \text{ is odd}$$ 
for every group $G$ and $S \subset G$.
In Proposition~\ref{prop: integral MDCGs} we give necessary and sufficient conditions for $MX^*(G;S,T)  \in  \mathcal{F}$ to be integral in terms of the  set $S$, depending on whether $G$ is a cyclic, an abelian or a non-abelian group. Then, in 
Examples~\ref{exam: circulant}--\ref{exam: Dn}, we give families of integral Cayley graphs $X(G,S)$. Using this, in Theorem~\ref{teo: cayley even/odd} we show that there are even and odd MDCGs defined over cyclic, abelian and non-abelian groups. These graphs are Cayley (sum) graphs over $G \times \Z_2$.

In Section \ref{sec: integral}, using the results of the previous section, we construct explicit examples of even and odd MDCGs both over finite fields and finite rings. 
First, by using known families of integral GP-graphs, that is graphs of the form $\G(k,q)=X(\ff_q, P_k)$ where $P_k=\{x^k:x\in \ff_q^*\}$, in Theorems \ref{thm: even/odd over Fqs} and \ref{thm: infinite families of even/odd graphs over Fqs} we exhibit infinite families of even and odd MDCGs graphs (in particular Cayley (sum) graphs). Namely, 
$$ MX^*(\ff_q;P_k,P_k) \text{ is even} \qquad \text{and} \qquad MX^*(\ff_q;P_k,P_k\cup \{0\}) \text{ is odd} $$ 
for some choices of $q$ and $k$.  
Then, using unitary Cayley (sum) graphs over ring, that is $X(R,R^*)$, where $R$ is a finite commutative ring with unit and $R^*$ the group of units, in Proposition \ref{prop: GRR integral} we show that 
$$ MX^*(R;R^*,R^*) \text{ is even} \qquad \text{and} \qquad MX^*(R;R^*,R^*\cup \{0\}) \text{ is odd}. $$

In Sections \ref{sec: crossed isosp} and \ref{sec: even/odd isosp} we study isospectrality problems for MDCGs in the family $\mathcal{F}$.
In Section~\ref{sec: isosp} we consider three different situations. First, in Proposition~\ref{prop: isospec T,T'} we show that there are no crossed isospectrality in $\mathcal{F}$. That is, for each fixed pair $(G,S)$, the pair of graphs $MX^*(G;S,T)$ and $MX^*(G;S,T')$ are not isospectral, for any pair of different $T,T' \in \mathcal{S}$ with $\mathcal{S} = \{\{e\}, S, S\cup \{e\}\}$.  Then, in Theorem \ref{thm: isosp X,X+}, we prove that the pair of mirror di-Cayley (sum) graphs 
    $$\{MX(G;S,T), MX^+(G;S,T)\}$$
    with $T \in \mathcal{S}$, 
are isospectral if and only if the pair of underlying Cayley (sum) graphs $\{X(G,S), X^+(G,S)\}$ are isospectral.
Finally, we consider the more general situation. For two pairs of groups and subsets $(G_1,S_1)$ and $(G_2,S_2)$, we prove that isospectrality between pairs 
\begin{gather*}
\{ MX^*(G_i;S_i,\{e\}) \}_{i=1,2}, \quad \{MX^*(G_i;S_i,S_i) \}_{i=1,2}, \\[1mm] 
\{MX^*(G_i;S_i,S_i \cup \{e\}) \}_{i=1,2}    
\end{gather*}
is possible if and only if the Cayley graphs $\{X^*(G_1,S_1), MX^*(G_2,S_2)\}$ are isospectral (see Theorem \ref{thm: gen isosp}).

The paper closes with Section \ref{sec: even/odd isosp}, where we construct explicit pairs of isospectral integral Cayley (sum) graphs $\{\G_E,{\G_E}^+\}$ both having even spectrum and $\{\G_O,{\G_O}^+\}$ both having odd spectrum. 
To perform this task, we consider mirror di-Cayley (sum) graphs defined over certain finite commutative rings with unity $R$. 
More precisely, in Proposition~\ref{prop: isosp R} we show that if $R$ is a finite commutative ring with identity having Artin's decomposition 
    $$ R=R_1 \times \cdots \times R_s,$$ 
where each $R_i$ is a local ring, with at least one $R_i$ of even size and at least one $R_j$ of odd size, then 
    $$ \{ MX(R;R^*,T), MX^+(R;R^*,T) \} $$ 
are isospectral for any $T \in \{ \{0\}, R^*, R^* \cup \{0\} \}$. 
Finally, putting together the previous results, we get our main result Theorem \ref{thm: main}, asserting that there exists isospectral pairs of integral Cayley (sum) graphs $\{ X(G,S), X^+(G,S)\}$, both having even symmetric spectrum or both having odd non-symmetric spectrum.

\section{Three families of mirror di-Cayley graphs} \label{sec: 3 families} 
For a group $G$ and subsets $S_\ell, S_r, T$ of $G$ one can define a \textit{di-Cayley graph} $DX(G;S_\ell, S_r, T)$ as a generalization of the bi-Cayley graph $BX(G;S_\ell, S_r, T)$, by allowing the edges defined by the connection set $T$ between the Cayley covers $X(G,S_\ell)$ and $X(G,S_r)$ to be directed (for the definition and properties of bi-Cayley (or semiregular) graphs we refer to \cite{GL1}, \cite{GL2}, \cite{KMS}, \cite{LF}, \cite{Marusic}, \cite{RJ}).
Then, a mirror di-Cayley graph $MX(G;S,T)$ is just a di-Cayley graph $DX(G;S_\ell, S_r, T)$ with equal connection sets $S_\ell=S_r=S$, that is with isomorphic covers.
More precisely, we have the following.
	
\begin{defi} \label{def: dicayleys}
Let $G$ be a group and $S,T \subset G$.
A \textit{mirror di-Cayley graph} 
    $$ MX(G;S,T) $$ is the graph having vertex set $G\times \Z_2$ and where there is a directed edge from $(h,i)$ to $(g,j)$ if and only if 
\begin{equation} \label{eq: dicays}
	j=i \quad \text{and} \quad gh^{-1} \in S \qquad \text{or} \qquad  j=i+1 \quad \text{and} \quad gh^{-1} \in T.    
\end{equation}
\textit{Mirror di-Cayley sum graphs} $MX^+(G;S,T)$ are defined similarly, changing $gh^{-1}$ by $gh$ above. 
\end{defi}

Throughout the paper, we will use the following.

\noindent
\textit{Notation:}  
($i$) When considering mirror di-Cayley graphs and mirror di-Cayley sum graphs together, we will denote them by $MX^*(G;S,T)$ and use the abbreviation MDCGs. Several times we will also use the abbreviation $\G_{G;S,T}^*$ for $MX^*(G;S,T)$.

($ii$) We say that $S$ is the \textit{connection} set and $T$ is the \textit{di-connection} set of $\G_{G;S,T}^*$. The isomorphic subgraphs $X^*(G,S) \times \{0\}$ and $X^*(G,S) \times \{1\}$ are called the \textit{mirrors} of $\G_{G;S,T}^*$. 

($iii$) Following the nomenclature for bi-Cayley graphs, we say that the graph $MX(G;S,\{e\})$ is the \textit{one-matching mirror di-Cayley graph.}

From now on, we will consider the three families of mirror di-Cayley (sum) graphs 
\begin{equation}
\begin{gathered}
\G^*_e=MX^*(G;S,\{e\}), \qquad \G^*_S=MX^*(G;S,S), \\
     \G^*_{S\cup e} = MX^*(G;S,S \cup \{e\}),    
\end{gathered}
\end{equation}
that is the graphs $MX^*(G;S,T)$ with di-connection set $T \in \mathcal{S}$ with 
\begin{equation} \label{eq: S}
    \mathcal{S} = \big\{ \{e\}, S, S\cup \{e\} \big\}. 
\end{equation}
Since $MX^*(G;S,S \cup \{e\}) = MX^*(G;S,S))$ if $e\in S$, from now on we will assume that $S$ is not a subgroup and that $e\notin S$ for the graph $MX^*(G;S,S \cup \{e\})$, unless explicit mention to the contrary.
We denote this family of MDCGs by $\mathcal{F}$. That is, 
\begin{equation} \label{eq: famili F}
    \mathcal{F} = \{ MX^*(G;S,T) \}_{T \in \mathcal{S}} 
\end{equation}
for $G$ a group, $S \subset G$ and $\mathcal{S}$ as in \eqref{eq: S}.

\subsection*{Basic properties}
Here we give the basic structural properties of the graphs in the family $\mathcal{F}$ considered. 
	
We will need the following definitions. A subset $S$ of a group $G$ is \textit{symmetric} if it is closed under inversion, 
i.e.\@ 
    $$ S=S^{-1},$$ 
(this holds in particular if $S$ is a subgroup); and $S$ is \textit{antisymmetric} if 
    $$ S \cap S^{-1} = \varnothing. $$ 
The subset $S$ is called \textit{normal} if for every $gh\in S$ we have that $hg \in S$. 
This happens if and only if 
    $$ gSg^{-1} =S \quad \text{for every $g\in G$},$$ 
i.e.\@ $S$ is closed under conjugation (hence, union of conjugacy classes). On the other hand, we say that $S$ is \textit{antinormal}, if 
    $$S \cap N_G(S) =\varnothing,$$
where 
    $N_G(S) = \{g\in G : gSg^{-1} =S \}$ is the normalizer of $S$ in $G$.
Of course, if $G$ is abelian then any $S$ is closed under conjugation.

The following statements about the directedness, loops, and the regularity degrees of the graphs in $\mathcal{F}$
can be easily deduced from the definitions.

\begin{prop} \label{prop: directedness} 
Let $G$ be a group and $S \subset G$. 
Consider the mirror di-Cayley graph $\G = MX(G;S,T)$ and the mirror di-Cayley sum graph $\G^+ = MX^+(G;S,T)$, where $T=\{e\}, S$ or $S\cup \{e\}$. Then, we have: \sk 
		
\noindent 
$(a)$ \textsc{Directedness}.
The graph $\G$ is undirected (resp.\@ directed) if and only if $S$ is symmetric (resp.\@ antisymmetric). 
The graph $\G^+$ is undirected (resp.\@ directed) if and only if $S$ is normal (resp.\@ antinormal). \sk 
		
\noindent 
$(b)$ \textsc{Loops}.
The graph $\G$ has loops at every vertex if and only if $e\in S$.
The graph $\G^+$ has loops at $(x,0)$ and $(x,1)$ if and only if $x^2\in S$. This happens for instance for the elements $x$ of order 2 in $S$ if $e\in S$.
In particular, if $S$ is a subgroup of $G$ then both $\G$ and $\G^+$ are looped. \sk 
		
\noindent 
$(c)$ \textsc{Regularity}.
The graphs $MX^*(G;S,\{e\})$ are $(s+1)$-regular, the graphs $MX^*(G;S,S)$ are $(2s)$-regular and the graphs $MX^*(G;S,S\cup\{e\})$ are $(2s+1)$-regular, where $|S|=s$.
\end{prop}

\noindent \textsc{Note}. 
Items $(a)$ and $(b)$ also hold if we consider the Cayley graphs $\G=X(G,S)$ and $\G^+=X^+(G,S)$.

\sk 

Notice that, for a fixed pair $(G,S)$, by comparing regularity degrees we have that the graphs in the three families 
$MX^*(G;S,\{e\})$, $MX^*(G;S,S)$ and $MX^*(G;S,S\cup\{e\})$ are all different, except for trivial exceptions when $|S|=0,1$: 

$\bullet$ For $S=\varnothing$, we have 
\begin{equation*} \label{eq: MGSe nP2}
 MX^*(G;S, \{e\}) = MX^*(G;S,S\cup\{e\})= MX^*(G; \varnothing, \{e\}) \simeq nP_2    
\end{equation*}
with $n=|G|$, where $nP_2 =P_2 \sqcup \cdots \sqcup P_2$ denotes $n$ disjoint copies of $P_2$. 

$\bullet$ For $S=\{e\}$, with $n=|G|$, we have  
\begin{gather*}
 MX(G;S, \{e\}) = MX(G;S,S) = MX(G;\{e\},\{e\}) = n \mathring{P}_2, \\
    MX^+(G;S, \{e\}) = MX^+(G;S,S) = MX^+(G;\{e\},\{e\}) \simeq  mC_4 \sqcup \ell \mathring{P}_2, 
\end{gather*}    
with 
    $$m=\#\{ g\in G : g=g^{-1}\} \qquad \text{and} \qquad \ell=\#\{ g\in G : g \ne g^{-1}\},$$ 
where $\mathring{P}_2$ denotes $P_2$ with one loop added to each vertex.

We now show that the MDCGs have a Cayley structure.

\begin{prop} \label{prop: cayley structure} 
Let $G$ be a group and $S \subset G$. 
Consider the mirror di-Cayley (sum) graph $MX^*(G;S,T)$ where $T=\{e\}, S$ or $S\cup \{e\}$. If $\Z_2=\{0,1\}$, then
\begin{equation} \label{eq: bicay=cay}
MX^*(G;S,T) = X^*\big( G\times \Z_2, (S\times \{0\}) \cup (T \times \{1\}) \big). 
\end{equation}
In particular, $MX^*(G;S,S) = X^* ( G \times \Z_2, S \times \Z_2 )$.
\end{prop}

\begin{proof}
In fact, the vertex set of both graphs in \eqref{eq: bicay=cay} is the same, $G \times \Z_2$. Relative to the edges,  let $g,h \in G$ and $i,j \in \{0,1\}$. 
For $MX(G;S,T)$ we have that 
\begin{gather*}
	(g,i) \sim (h,i) \quad \Leftrightarrow \quad (hg^{-1},0) \in S\times \{0\}, \\ 
	(g,i) \sim (h,j) \quad \Leftrightarrow \quad (hg^{-1},1) \in T \times \{1\}, 
\end{gather*}
for $i\ne j$; while for $MX^+(G;S,T)$ we have to change $g^{-1}$ by $g$ above.
Thus, the edge sets of the mirror di-Cayley (sum) graphs $MX^*(G;S,T)$ coincides with the edge sets of the Cayley (sum) graphs $X^*\big( G\times \Z_2, (S\times \{0\}) \cup (T \times \{1\}) \big)$, correspondingly.
\end{proof}

\subsection*{Product decompositions}
We next show that the mirror di-Cayley (sum) graphs $MX^*(G;S,T)$ considered can be decomposed as simple products of the underlying Cayley graph $X^*(G,S)$.

There are many graph products in the literature. The most common are the \emph{Cartesian product} $(\Box)$, the \emph{direct product} $(\times)$, and the \emph{strong product} $(\boxtimes)$, all defined on the Cartesian product of vertex sets. Given two graphs $\Gamma_1$ and $\Gamma_2$, their product is a graph with vertex set 
    $$V(\Gamma_1) \times V(\Gamma_2),$$ 
and adjacency defined as follows:
	
$\bullet$ In the \emph{Cartesian product} $\Gamma_1 \Box \Gamma_2$, vertices $(u_1, u_2)$ and $(v_1, v_2)$ are adjacent if either $u_1 = v_1$ and $u_2v_2 \in E(\Gamma_2)$ or else $u_2 = v_2$ and $u_1v_1 \in E(\Gamma_1)$.
	
$\bullet$ In the \emph{direct product} $\Gamma_1 \times \Gamma_2$, adjacency holds between $(u_1, u_2)$ and $(v_1, v_2)$ when $u_1v_1 \in E(\Gamma_1)$ and $u_2v_2 \in E(\Gamma_2)$. This product is also known as Kronecker $(\otimes)$ product.
	
$\bullet$ The \emph{strong product} $\Gamma_1 \boxtimes \Gamma_2$ combines the previous two:
\begin{equation} \label{eq: strong product}
	E(\Gamma_1 \boxtimes \Gamma_2) = E(\Gamma_1 \Box \Gamma_2) \cup E(\Gamma_1 \times \Gamma_2).
\end{equation}

All three products are \emph{commutative} and \emph{associative}.
The products are well defined for any kind of simple graphs, that is directed, undirected or of mixed type.

We next show that the mirror di-Cayley (sum) graphs $MX^*(G;S,T)$ considered can be decomposed as products between the associated Cayley graph $X^*(G,S)$ and 2-path graphs.
	
\begin{thm} \label{teo: prods}
Let $G$ be a group and $S$ a subset of $G$. We have that
	\begin{equation} \label{eq: prod}
		\begin{aligned}
			MX^{*}(G;S,\{e\}) 		&= X^{*}(G,S) \Box P_2,\\
			MX^{*}(G;S,S) 			&= X^{*}(G,S)\times \mathring{P_2},\\
			MX^{*}(G;S,S\cup\{e\})  &= X^{*}(G,S) \boxtimes P_2,
		\end{aligned}
	\end{equation}
where $P_2$ and $\mathring{P_2}$ are the $2$-path and the looped $2$-path, respectively.
\end{thm}

\begin{proof}
Let $\Gamma_1=X^*(G,S)$ and $\Gamma_2=P_2$ or $\mathring{P_2}$. Note that 
	$$ V(X^*(G,S)) = G \qquad \text{and} \qquad V(P_2) = V(\mathring{P_2}) = \{0,1\}. $$
Then, the vertex group of each product is $G\times\{0,1\}$, which coincides with the set of vertices of $MX^*(G;S,T)$ for any $T \in \{\{e\}, S, S\cup\{e\} \}$.
		
Now let us see, in each case, when two vertices are adjacent. Recall that
	$$ E(X(G,S)) = \{uv\,:\,vu^{-1}\in S\} \qquad \text{and} \qquad  E(X^+(G,S)) = \{uv\,:\,vu \in S\}$$
and that $E(P_2) = \{01\}$ and $E(\mathring{P_2}) = \{00,01,11\}$.
		
In the first case, $(h,i) \sim (g,j)$ in $MX(G;S,\{e\})$ (resp.\@ $MX^+(G;S,\{e\})$), with $h,g \in G$ and $i,j \in \{0,1\}$, if and only if 
	$$ i = j \:\: \text{and} \:\: gh^{-1} \in S \: (\text{resp. } gh \in S) \quad \text{or} \quad i \ne j \:\: \text{and} \:\: gh^{-1} \in \{e\} \: (\text{resp. } gh \in \{e\}).$$
Then, we get that $E(MX^*(G;S,\{e\})) = E(X^*(G,S)\Box P_2)$ and hence the graph
$MX^*(G;S,\{e\})$ equals $X^*(G,S)\Box P_2$.
		
In the second case, 
$(h,i) \sim (g,j)$ in $MX(G;S,S)$ (resp.\@ $MX^+(G;S,S)$), with $h,g \in G$ and $i,j \in \{0,1\}$, if and only if 
		$$ i = j \:\: \text{and} \:\: gh^{-1} \in S \: (\text{resp. } gh \in S) \quad \text{or} \quad i \ne j \:\: \text{and} \:\: gh^{-1} \in S \: (\text{resp. } gh \in S).$$
Then, we have that $E( MX^*(G;S,S)) = E(X^*(G,S)\times \mathring{P_2})$ and thus $MX^*(G;S,S) = X^*(G,S) \times \mathring{P_2}$.
		
Finally, 
$(h,i) \sim (g,j)$ in $MX(G;S,S \cup \{e\})$, 
with $h, g \in G$ and $i,j \in \{0,1\}$, if and only if 
	$$ i = j \:\: \text{ and } \:\: gh^{-1} \in S \cup \{e\} 
		\qquad \text{or} \qquad i \ne j \:\: \text{ and } \:\: gh^{-1} \in S \cup \{e\},$$
and the same holds for $MX^+(G;S,S \cup \{e\})$ changing $gh^{-1} \in S \cup \{e\}$ by $gh \in S \cup \{e\}$ above.

Then, the set of edges are equal, $E(MX^*(G;S,S \cup \{e\})) = E(X^*(G,S) \boxtimes P_2)$, and hence $MX^*(G;S,S \cup \{e\}) = X^*(G,S) \boxtimes P_2 $. 
The proof of \eqref{eq: prod} is thus complete.  
\end{proof}

In the next remark, we continue the study of the relation between the MDCGs in the family $\mathcal{F}$ and the product decompositions with 2-paths. We will need the following notation. 
For any graph $\G$ we will denote by $\mathring{\G}$ the graph $\G$ with a loop at every vertex, that is 
    $$ \mathring{\G} = \G \cup \mathring{0}_n ,$$ 
where $\mathring{0}_n$ is the empty graph with $n=|G|$ vertices and a loop at every vertex. 

Also, notice that 
	\begin{equation} \label{eq: unions}
		\begin{split}
			MX^*(G;S,T \cup T') &= MX^*(G;S,T) \cup MX^*(G;S,T'), \\
			MX^*(G;S \cup S', T) &= MX^*(G;S,T) \cup MX^*(G;S',T),    
		\end{split}
	\end{equation}
for any $S,S',T,T'$ subsets of a group $G$. 
	
\begin{rem} \label{rem: other products}
We now give the products of $X^*(G,S)$ with $P_2$ or $\mathring{P}_2$ not appearing in Theorem \ref{teo: prods}.

\noindent ($i$)
It is not difficult to check that 
	$$  X^*(G,S)\times P_2 = MX^*(G;\varnothing, S)$$
which is a mirror di-Cayley graph, although not in the family $\mathcal{F}$, 
and that 
	$$ X^*(G,S) \Box \mathring{P}_2 = \begin{cases}
		X^*(G,S) \Box P_2 & \qquad \text{if $e\in S$}, \\[1mm] 
		(X^*(G,S) \Box P_2)^\circ & \qquad \text{if $e\notin S$}.   
    \end{cases}$$
Also, using the definition of $\boxtimes$ and the previous results, one can see that  
		$$ X^*(G,S) \boxtimes  \mathring{P}_2 = \begin{cases}
			MX^*(G;S,S) & \qquad \text{if $e\in S$}, \\[1mm] 
			(MX^*(G;S,S\cup \{e\})^\circ & \qquad \text{if $e\notin S$}.   
		\end{cases}$$
		
\noindent ($ii$)
We can also obtain the looped versions of the graphs in $\mathcal{F}$: 
	$$  MX^*(G;S,S)^\circ = (X(G,S) \times \mathring{P}_2)^\circ, \qquad MX^*(G;S,\{e\})^\circ =
		(X(G,S) \Box \mathring{P}_2)^\circ $$
and we saw that $(MX^*(G;S,S\cup \{e\})^\circ = X^*(G,S) \boxtimes  \mathring{P}_2$ for $e\notin S$.
We will denote by $\mathring{\mathcal{F}}$ the family of graphs in $\mathcal{F}$ with loops.
\end{rem}

\subsection*{NEPS of graphs} 
A more general definition of the above classic products is given by the concept of NEPS. 	
Let $\mathcal{B}$ be a set of binary non-zero $n$-tuples, i.e.\@ 
	$$\mathcal{B} \subset\{0,1\}^n \smallsetminus \{(0,\dots,0)\}$$ 
such that for every $i = 1,\dots, n,$ there exist $\beta\in\mathcal{B}$ with $\beta_i=1$. The \textit{non-complete extended p-sum (NEPS)} of $n$ graphs $\G_1,\dots,\G_n$ with basis $\mathcal{B}$, denoted by 
	$$\text{NEPS}(\G_1,\dots,\G_n;\mathcal{B}),$$ 
is the graph with the vertex set 
    $$ V(\G_1)\times \cdots \times V(\G_n),$$ 
in which two vertices $(u_1,\dots,u_n)$ and $(v_1,\dots,v_n)$ are adjacent if and only if there exist $(\beta_1,\dots,\beta_n)\in\mathcal{B}$ such that $u_i$ is adjacent to $v_i$ in $\G_i$ whenever $\beta_i=1$ and $u_i=v_i$ whenever $\beta_i=0$.

\goodbreak 
For $n=2$, we have the following instances of $\text{NEPS}(\G_1,\G_2; \mathcal{B})$:
\begin{itemize}
	\item the \textit{Cartesian product} $\G_1\Box \G_2$, when $\mathcal{B}=\{(1,0),(0,1)\}$; \sk 
		
	\item the \textit{direct product} $\G_1 \times \G_2$ when, $\mathcal{B}=\{(1,1)\}$; \sk 
		
	\item the \textit{strong product} $\G_1 \boxtimes \G_2$ when, $\mathcal{B}=\{(1,0),(0,1),(1,1)\}$; and \sk  
		
	\item the \textit{strong sum}  $\G_1 \oplus \G_2$ when, $\mathcal{B}=\{(1,0),(1,1)\}$ (see \cite{Stevanovic}).
\end{itemize}
	
In contrast to the Cartesian, direct, and strong products, which are all commutative and associative, the strong sum preserves associativity but fails to be commutative in general. The strong sum has not to be confused with the ringsum of graphs, also denoted by $\oplus$. The \textit{ring sum} $G_1 \oplus G_2$ of the graphs $G_1=(V_1, E_1)$, $G_2=(V_2, E_2)$ is the graph with vertex set $V_1 \cup V_2$ and edges set $E_1 \Delta E_2 = (E_1 \cup E_2) - (E_1\cap E_2)$.

Relative to di-Cayley graphs, Theorem \ref{teo: prods} shows that the product of the Cayley graph $X^*(G,S)$ with the $2$-path graph $P_2$ (or with the $2$-path with loops $\mathring{P_2}$) is a mirror di-Cayley graph for the first three NEPS of two graphs. 
In the case of the strong sum of graphs we have the following.  

\begin{lem} \label{lem: strong sum}
For any group $G$ and a subset $S \subset G$ we have:
\begin{equation} \label{eq: strong sum}
		\begin{aligned}
			X^{*}(G,S) \oplus P_2 &= MX^{*}(G;S,S), \\[.5mm]
			P_2 \oplus X^{*}(G,S) &= MX^{*}(G;\varnothing,S \cup\{e\}).
		\end{aligned}
	\end{equation} 
    
\end{lem}	
	
\begin{proof}
By definition, two vertices $(u_1,u_2)$ and $(v_1,v_2)$ are adjacent in the strong sum if and only if $u_1$ is adjacent to $v_1$ and either $u_2=v_2$ or $u_2$ is adjacent to $v_2$. 

To prove the first equality in \eqref{eq: strong sum}, notice that, in $X^{*}(G,S) \oplus P_2$, two vertices $(u_1,u_2)$ and $(v_1,v_2)$ are adjacent if and only if $u_1$ and $v_1$ are adjacent in $X^*(G,S)$, then 
$X^{*}(G,S) \oplus P_2 = MX^{*}(G;S,S)$.
For the second identity, 
two vertices $(u_1,u_2)$ and $(v_1,v_2)$ are adjacent in $P_2 \oplus X^{*}(G,S)$ if and only if $u_1\neq v_1$, then
$P_2 \oplus X^{*}(G,S) = MX^{*}(G;\varnothing,S \cup\{e\})$.
\end{proof}

By \eqref{eq: prod} and \eqref{eq: strong sum} we have two decompositions of the mirror di-Cayley graph $MX^{*}(G;S,S)$ into different products between $X^*(G,S)$ and $P_2$ or $\mathring{P_2}$: 
\begin{equation} \label{eq: XGSx+P2}
    MX^{*}(G;S,S) = X^{*}(G,S) \times \mathring{P_2} = X^{*}(G,S) \oplus P_2.     
\end{equation}	

Summing up we have the following.
\begin{coro}
Any mirror di-Cayley (sum) graph $MX^*(G;S,T)$ with $T \in \mathcal{S}$, where 
$ \mathcal{S}= \{ \{e\}, S, S\cup \{e\}\}$, 
can be obtained as a NEPS:
\begin{enumerate}[$(a)$]
    \item of $X^*(G,S)$, using the four NEPS of two graphs: $\square$, $\times$, $\Box$ and $\oplus$; \msk 
    
    \item of $X^*(G,S)$ and $P_2$. 
\end{enumerate}
\end{coro}

\begin{proof}
It is immediate by Theorem \ref{teo: prods} and Lemma \ref{lem: strong sum}. 
\end{proof}

\section{Spectrum} \label{sec: spec}
In this section, we compute the spectrum of the mirror di-Cayley (sum) graphs in the family $\mathcal{F}$ defined in \eqref{eq: S}--\eqref{eq: famili F}. 
We first use the product decomposition of the these graphs to put the spectra of $MX^*(G;S,T)$, with $T \in \mathcal{S}$, in terms of the spectrum of $X^*(G,S)$, and then use the known spectrum of these Cayley (sum) graphs to express the spectra of $MX^*(G;S,T)$, with $T \in \mathcal{S}$, in terms of irreducible characters of $G$.

\subsection*{The spectrum via products}
Now, using the product structure of the graphs in $\mathcal{F}$, we give the spectrum of $MX^*(G;S,T)$ for $T=\{e\}, S$ and $S \cup \{e\}$ in terms of the spectrum of $X^*(G,S)$.
	
\begin{prop} \label{prop: spec bicayleys} 
Let $G$ be a group and $S \subset G$. If $\text{Spec}(X^{*}(G,S))=\{[\lambda_i^*]^{m_i^*}\}_{i\in I}$ then
\begin{equation} \label{eq: specbi}
\begin{aligned}
Spec(MX^{*}(G;S,\{e\}))      &= \{[\lambda_i^{*} + 1]^{m_i^*},[\lambda_i^*-1]^{m_i^*}\}_{i\in I},\\
Spec(MX^{*}(G;S,S))          &= \{[2\lambda_i^{*}]^{m_i^*}\}_{i\in I} \cup \{[0]^{|G|}\}, \\
Spec(MX^{*}(G;S,S\cup\{e\})) &= \{[2\lambda_i^{*}+1]^{m_i^*}\}_{i\in I} \cup \{[-1]^{|G|}\}.
\end{aligned}
\end{equation}
\end{prop}
	
\begin{proof}
It is well-known that (see, for instance, \cite{BK}), given graphs $\G_1$ and $\G_2$, if $Spec(\Gamma_1)=\{[\lambda_i]^{m_i}\}_{i \in I}$ and $Spec(\Gamma_2)=\{[\mu_j]^{n_j}\}_{j \in J}$, then we have 
\begin{equation} \label{spec}
	\begin{aligned}
		Spec(\Gamma_1\Box\Gamma_2) &= \{[\lambda_i + \mu_j]^{m_i n_j} \}_{(i,j) \in I \times J}, \\[1mm]
		Spec(\Gamma_1 \times \Gamma_2) &= \{[\lambda_i \mu_j]^{m_i n_j}\}_{(i,j) \in I \times J}, \\[1mm]
		Spec(\Gamma_1 \boxtimes \Gamma_2) &= \{[\lambda_i+\mu_j+\lambda_i \mu_j]^{m_i n_j}\}_{(i,j) \in I \times J}.
	\end{aligned}
\end{equation}
	
By using this, and putting 
$\Gamma_1=X^*(G,S)$ and $\Gamma_2=P_2$ (respectively $\mathring{P_2}$), we can obtain the spectrum of the products. We know that 
	$$Spec(P_2)=\{[1]^1,[-1]^1\} \qquad \text{and} \qquad Spec(\mathring{P_2})=\{[0]^1,[2]^1\}.$$ 
Then, by \eqref{eq: prod} and \eqref{spec}, we have that
\begin{equation*} \label{eq: Specs 1y2}
	\begin{aligned}
		Spec(MX^*(G;S,\{e\})) &= Spec(X^*(G,S) \Box P_2)
		= \{[\lambda_i^* + 1]^{m_i^*},[\lambda_i^*-1]^{m_i^*}]\}_{i\in I},\\
		Spec(MX^*(G;S,S))     &= Spec(X^*(G,S)\times \mathring{P_2})
		= \{[2\lambda_i^*]^{m_i^*},[0]^{m_i^*}\}_{i\in I},
	\end{aligned}
\end{equation*}	
and 
    $S pec(MX^*(G;S,S\cup\{e\})) = Spec(X^*(G,S)\boxtimes P_2)$, 
which equals 
    $$ \{ [\lambda_i^*+1+\lambda_i^*]^{m_i^*},[\lambda_i^*+(-1)+ (-1)\lambda_i^*]^{m_i^*}\}_{i\in I} = \{[2\lambda_i^*+1]^{m_i^*},[-1]^{m_i^*}\}_{i\in I}.$$

Finally, notice that the eigenvalue $0$ in $MX^*(G;S,S)$ appears with multiplicity 
    $$ m^*(0) = \sum_{i\in I} m_i^* = |G|. $$ 
Similarly, the multiplicity of the eigenvalue $-1$ in $MX^*(G;S,S\cup\{e\})$ is given by 
    $$m^*(-1)=|G|.$$
In this way, we have obtained the expression \eqref{eq: specbi} for the spectra of the graphs.
\end{proof}

\begin{rem}
By the previous proposition, the spectra of the four graphs involved determine one another. 
More precisely, if we know $Spec(X^{*}(G,S))$ we can deduce $Spec(MX^{*}(G;S,T))$ for any $T \in \{S,\{e\},S\cup \{e\}\}$. Conversely, if we know $Spec(MX^{*}(G;S,T))$ for one $T$, then we can deduce $Spec(X^{*}(G,S))$ and $Spec(MX^{*}(G;S,T'))$ for $T' \in \{S,\{e\},S\cup \{e\}\}$ with $T' \ne T$.
\end{rem}

\begin{rem}
By Theorem $2.23$ of \cite{Cvetkovic}, the spectrum of the strong sum $\G_1\oplus \G_2$ is 
    $$ Spec(\G_1\oplus \G_2) = \{[\lambda_i+\lambda_i\mu_j]^{m_i\cdot n_j}\}_{(i,j)\in I\times J},$$ 
where $Spec(\G_1)=\{[\lambda_i]^{m_i}\}_{i \in I}$ and $Spec(\G_2)=\{[\mu_j]^{n_j}\}_{j \in J}$. 
Now, since $MX^*(G;S,S) = X^*(G,S) \oplus P_2$, the spectrum of $MX^*(G;S,S)$ is given by 
	$$ Spec(MX^*(G;S,S)) =  \{ [2\lambda^*_i]^{m_i^*}, [0]^{m_i^*}\}_{i\in I} = \{ [2\lambda^*_i]^{m_i^*} \}_{i\in I} \cup \{[0]^{|G|}\},$$ 
where $Spec(X^*(G,S))=\{[\lambda^*_i]^{m_i^*}\}_{i\in I}$, which coincides with the previous proposition. 
\end{rem}

\subsection*{The spectrum in terms of characters}
Proposition \ref{prop: spec bicayleys} gives the spectrum of the mirror di-Cayley (sum) graphs $MX^*(G;S,T)$ with $T\in \mathcal{S}$ in terms of the spectrum of the associated Cayley (sum) graph $X^*(G,S)$. 
Fortunately, the spectrum of the Cayley graph $X(G,S)$ and the Cayley sum graph $X^+(G,S)$ are known, for $S$ a normal 
subset ($S$ symmetric is also required for the sum graph), in terms of irreducible characters of $G$.

We will need the following standard notation. 
For $S$ a subset of $G$ and $\chi$ a character of $G$, we put
\begin{equation} \label{eq: chiS}
    \chi(S) =\sum_{s\in S}\chi(s).    
\end{equation} 
	
The following is a very well-known result.

\begin{lem} \label{lem: eigenvalues X*GS}
Let $G$ be a finite group and $S$ a normal subset of $G$. Then, we have:

\begin{enumerate}[$(a)$]
    \item The eigenvalues of the Cayley graph $\Gamma=X(G,S)$ are given by
			$$ {\rm Eig}(\G) = \left\{ \lambda_\chi(\G) = \frac{\chi(S)}{\chi(1)} : \chi \in \hat{G} \right\}.$$
			
    \item The eigenvalues of the Cayley sum graph $\G^+=X^+(G,S)$ are real and given by
			$$ {\rm Eig}(\G^+) = \{ \lambda^\pm_\chi(\G^+) = \pm |\chi(S)| : \chi\in\hat{G} \}, $$ 
			if $G$ is abelian, and by 
			$$ {\rm Eig}(\G^+) = \left\{ \lambda^\pm_\chi(\G^+) =\pm \frac{\chi(S)}{\chi(1)} : \chi \in \hat G \right \},$$ 
			if $G$ is not abelian and $S$ is also symmetric.
\end{enumerate}
In both cases, $\chi(S)$ is as defined in \eqref{eq: chiS}.
\end{lem}

Recall that $\chi(1)= \dim V_\chi$.
If $G$ is abelian then $\chi(1)=1$ for any $\chi$ and every subset $S$ of $G$ is normal.

We are now in a position to give the spectrum of the MDCGs in the family $\mathcal{F}$. We do this in two separate results. 
We begin with the spectrum of mirror di-Cayley graphs in $\mathcal{F}$. 	
\begin{thm} \label{thm: spec 3fam chars}
Let $G$ be a group and $S$ a normal subset of $G$. 
Then, the spectra of the di-Cayley graphs in $\mathcal{F}$ are given by: 
\begin{equation} \label{eq: spec di chars}
	\begin{aligned}
		Spec(MX(G;S,\{e\})) &= \left\{ \left[  \tfrac{\chi(S)}{\chi(1)} \pm 1 \right]^{m_\chi} \right\}_{\chi \in \hat G},\\[1mm]
				Spec(MX(G;S,S)) &= \left\{ \left[ 2\tfrac{\chi(S)}{\chi(1)} \right]^{m_\chi} \right\}_{\chi \in \hat G} \cup \{[0]^{|G|}\}, \\[1mm]
				Spec(MX(G;S,S\cup\{e\})) &= 
				\left\{ \left[ 2\tfrac{\chi(S)}{\chi(1)}+1 \right]^{m_\chi} \right\}_{\chi \in \hat G} \cup \{[-1]^{|G|}\},
		\end{aligned}
\end{equation}    
\nopagebreak 
where $m_\chi$ is the multiplicity of the eigenvalue $\lambda_\chi$ in each corresponding graph. 
\end{thm}

\goodbreak 

\begin{proof}
Just put together Proposition \ref{prop: spec bicayleys} and item ($a$) in Lemma \ref{lem: eigenvalues X*GS}. 	
\end{proof}
	
Now we give the spectrum of mirror di-Cayley sum graphs in $\mathcal{F}$.

\begin{thm} \label{thm: spec 3fam chars sum}
Let $G$ be a group and $S$ a normal subset of $G$. Then, the spectrum of the di-Cayley sum graphs in $\mathcal{F}$ is real and given by: 

\begin{enumerate}[$(a)$]
	\item If $G$ is abelian, then we have  
		\begin{equation} \label{eq: spec di+ ab}
			\begin{aligned}
				Spec(MX^{+}(G;S,\{e\})) &= \left\{ \left[  \pm |\chi(S)| +1 \right]^{m_\chi}, \left[  \pm |\chi(S)| -1 \right]^{m_\chi} \right\}_{\chi \in \hat G}, \\ 
				Spec(MX^{+}(G;S,S)) &= 
				\{[\pm 2|\chi(S)| ]^{m_\chi}\}_{\chi \in \hat G} \cup \{[0]^{|G|}\}, \\
				Spec(MX^{+}(G;S,S\cup\{e\})) &= 
				\{[\pm 2|\chi(S)|+1]^{m_\chi}\}_{\chi \in \hat G} \cup \{[-1]^{|G|}\}.
			\end{aligned}
		\end{equation}    
			
	\item If $G$ is not abelian and $S$ is also symmetric, then we have
		\begin{equation} \label{eq: spec di+ non-ab}
			\begin{aligned}
				Spec(MX^{+}(G;S,\{e\})) &= \left\{ \left[  \pm \tfrac{\chi(S)}{\chi(1)} +1 \right]^{m_\chi}, \left[  \pm \tfrac{\chi(S)}{\chi(1)}-1 \right]^{m_\chi} \right\}_{\chi \in \hat G}, \\ 
				Spec(MX^{+}(G;S,S)) &= 
				\left\{ \left[ \pm 2\tfrac{\chi(S)}{\chi(1)} \right]^{m_\chi} \right\}_{\chi \in \hat G} \cup \{[0]^{|G|} \}, \\
				Spec(MX^{+}(G;S,S\cup\{e\})) &= 
				\left\{ \left[ \pm 2 \tfrac{\chi(S)}{\chi(1)} +1 \right]^{m_\chi} \right\}_{\chi \in \hat G} \cup \{[-1]^{|G|}\}.
			\end{aligned}
		\end{equation}   
	\end{enumerate}
In both items $(a)$ and $(b)$, $m_\chi$ denotes the multiplicity of the eigenvalue $\lambda_\chi$ in each corresponding graph. 
\end{thm}
	
\begin{proof}
It follows directly from Proposition \ref{prop: spec bicayleys} and item ($b$) in Lemma \ref{lem: eigenvalues X*GS}. 
\end{proof}
	
Let $\G$ be any loopless graph and $A$ its adjacency matrix. If $\mathring{\G}$ is the looped  version of $\G$, then its adjacency matrix is given by $\mathring{A}=A+I$, and hence the eigenvalues of $\mathring{\G}$ are the eigenvalues of $\G$ plus 1, that is 
	$$Spec(\mathring{\G})=Spec(\G)+1.$$ 
Hence, from Theorem \ref{thm: spec 3fam chars}, we can easily obtain the spectrum of the graphs in $\mathring{\mathcal{F}}$.

\section{Even and odd integral graphs} \label{sec: integral}
In this section, we obtain explicit examples of integral MDCGs. In particular, we obtain examples of \textit{even graphs} (all even eigenvalues) and of \textit{odd graphs} (all odd eigenvalues), which are also Cayley graphs. For general results on integral Cayley graphs we refer to the survey paper of Liu-Zhou \cite{LZ}.

To better study integral spectrum of graphs, we introduce the following refinement.  
\begin{defi} \label{defi: even/odd}
If $\G$ is an integral graph, we define the \textit{even part} and the \textit{odd part} of $Spec(\G)$, respectively, as the sets of even eigenvalues and the set of odd eigenvalues of $\G$, correspondingly denoted by $Spec_{even}(\G)$ and $Spec_{odd}(\G)$. 
We say that $\G$ is \textit{even} (resp.\@ \textit{odd}) if $Spec(\G$) is even (resp.\@ odd), in symbols 
\begin{equation}
\begin{tabular}{ccccc}
    $\G \text{ is even }$ & \quad $\Leftrightarrow$ &\quad  $Spec(\G) \subset 2\Z$  & \quad  $\Leftrightarrow$ & \quad  $Spec_{odd}(\G) = \varnothing$, \\[2mm]
    $\G \text{ is odd }$ & \quad $\Leftrightarrow$ & \quad   $Spec(\G) \subset 2\Z+1$ & \quad $\Leftrightarrow$ & \quad  $Spec_{even}(\G) = \varnothing$.
 \end{tabular}
\end{equation}
\end{defi}

We begin by giving two observations which are direct consequences of the expressions for the spectrum of the MDCGs obtained in the previous section. The first one has to do with the nature of the spectrum. 

\begin{coro} \label{coro: integral}
Let $G$ be a group and $S$ a subset of $G$.

\noindent $(a)$ 
For any $T \in \{ \{e\}, S, S\cup\{e\} \}$ and $\mathcal{N} \in \{\mathbb{Z}, \mathbb{R}, \mathbb{C}\}$ we have that  
    $$Spec(MX^{*}(G;S,T)) \subset \mathcal{N} 
       \qquad \Leftrightarrow \qquad 
       Spec(X^{*}(G,S))\subset \mathcal{N}. $$ 
In particular, $MX^{*}(G;S,T)$ is integral if and only if $X^{*}(G,S)$ is integral. \msk 

\noindent $(b)$ 
In the case of integral spectrum, if $X^*(G,S)$ is integral we have: \sk 
 
$(i)$ 
$MX^{*}(G;S,\{e\})$ is even (resp.\@ odd) if and only $X^*(G,S)$ is odd (resp.\@ even). 

$(ii)$ 
$Spec(MX^{*}(G;S,S))$ is even and $Spec(MX^{*}(G;S,S \cup \{e\}))$ is odd. \sk 

Here, it is to be understood that the statements are either for $X(G,S)$ and $MX(G;S,T)$ or for $X^+(G,S)$ and $MX^+(G;S,T)$, respectively.
\end{coro}

\begin{proof}
Straightforward from Proposition \ref{prop: spec bicayleys}.
\end{proof}

\begin{rem} \label{rem: X int X*}
By Lemma \ref{lem: eigenvalues X*GS}, for $G$ abelian or, else, $G$ arbitrary but $S$ normal, we know that $X(G,S)$ is integral if and only if $X^+(G,S)$ is integral. Thus, in these cases, we only need to require that $X(G,S)$ (or $X^+(G,S)$) be integral in $(b)$ of Corollary~\ref{coro: integral}.    
\end{rem}

\begin{rem} \label{rem: ring of integers}
Notice that by the form of the eigenvalues of $X^*(G,S)$ obtained in Proposition \ref{prop: spec bicayleys}, item ($a$) of the corollary can be generalized. In fact, if $R$ is any subring of $\C$ then we have that  
$$ Spec(X^{*}(G,S))\subset R \qquad \Rightarrow \qquad Spec(MX^{*}(G;S,T)) \subset R.$$
In particular, this holds for $R$ equal to any number field $K$ (i.e.\@ $\Q \subset K \subset \C$) or to its ring of integers $\mathcal{O}_K$, such as $\Z[i]$ for example. 
\end{rem}

The second observation, related with the symmetry of the  spectrum, shows that there is a significative difference between the graphs with $T=\{e\}$ or $S$ and those with $T= S \cup \{e\}$.

\begin{coro} \label{coro: symmetric}
Let $G$ be a group and $S$ be a subset of $G$.
Then, the following holds:
\begin{enumerate}[$(a)$]
    \item $Spec(X^{*}(G,S))$ is symmetric if and only if $Spec(MX^{*}(G;S,S'))$ is symmetric with $S'=\{e\}$ or $S$.  \msk 

    \item If $Spec(X^{*}(G,S))$ is symmetric and $e\notin S$, then $Spec(MX^{*}(G;S,S \cup \{e\}))$ is not symmetric. 
\end{enumerate}
\end{coro}

\begin{proof}
For ($a$), this is clear from the expressions for the spectra in Proposition~\ref{prop: spec bicayleys}.

For ($b$), notice that we take $e\notin S$ for if not, the graph is symmetric by $(a)$, since $MX^*(G,S,S\cup \{e\})=MX^*(G,S,S)$, in this case. Thus, assume that $Spec(X^{*}(G,S))$ is symmetric and suppose by contradiction that $Spec(MX^{*}(G;S,S \cup \{e\}))$ is symmetric. Clearly, $-2\lambda_i-1 \ne 2\lambda_i+1$ since $\lambda_i = -\frac 12$ is not possible
(because there are no eigenvalues in $\Q\smallsetminus \Z$ for any graph), hence for every $i$ there is a $j \ne i$ such that     $$ -2\lambda_i-1 = 2\lambda_j+1.$$ 
But this implies that $\lambda_j=-\lambda_i-1$.
In particular, if we apply this to $\lambda_i=\lambda_1=k$ with $k=|S|$, we have that $\lambda_j=-(k+1)$ is an eigenvalue. But this is absurd, since the eigenvalues of a $k$-regular graph all lie in $[-k,k]$.
\end{proof}

Next, we illustrate the results of the two previous corollaries.
\begin{exam} \label{exam: CayZ4}
Let $G=\Z_4=\{0,1,2,3\}$ and $S=\Z_4^* = \{1,3\} \subset \Z_4$. The unitary Cayley graph $\G=X(G,S)$ equals the 4-cycle $C_4$ which has spectrum given by
    $$ Spec(C_4) = \{ [2]^1, [0]^2, [-2]^1 \}.$$
Then, by Proposition \ref{prop: spec bicayleys}, we have that
\begin{align*}
Spec(\G_e) & = \{ [3]^1,[1]^3,[-1]^3,[-3]^1 \}, \\ 
Spec(\G_S) & = \{ [4]^1, [0]^6, [-4]^1\}, \\ 
Spec(\G_{S \cup \{e\}}) & = \{[5]^1, [1]^2, [-1]^4,[-3]^1\} ,    
\end{align*}   
where $\G_e = MX(G;S,\{e\})$, $\G_S = MX(G;S,S)$ and $\G_{S\cup \{e\}} = MX(G;S,S\cup \{e\})$.
We see that $\G$ and $\G_S$ are even graphs while $\G_e$ and $\G_{ S \cup \{e\}}$ are odd graphs. Also, $\G$, $\G_e$ and $\G_S$ are symmetric but $\G_{S \cup \{e\}}$ is not symmetric.
\hfill $\diamond$
\end{exam}

We now make some remarks on even and odd graphs.

\begin{rem} \label{rem: even odd examples}
($a$) 
While there are many examples of integral graphs in the literature, there are not so many examples having either all even eigenvalues or, else, all odd eigenvalues (hence, this is one of our main motivations for studying the graphs in the family $\mathcal{F}$). Namely, we have the following trivial examples of parametric families:

\begin{enumerate}[$(i)$]
    \item The complete graph $K_n$ has eigenvalues $\{n-1,-1\}$ and, hence, odd spectrum for $n$ even.

    \item The star $K_{1,n}$ has eigenvalues $\{\pm \sqrt n, 0\}$ and, hence, even spectrum for all even square $n=4m^2$.

    \item The complete bipartite graph $K_{n,n}$ has eigenvalues $\{\pm n, 0\}$ and, hence, even spectrum for even $n$.

    \item The hypercube $Q_{n}$ has eigenvalues $\{n-2k : 0\le k\le n\}$ and, hence, even (resp.\@ odd) spectrum for $n$ even (resp.\@ odd).
\end{enumerate}

\noindent ($b$) 
As a non-trivial example of even graphs we can mention the Csikvari's trees. In \cite{Csi}, for any set of $k$ positive integers $r_1 < r_2 < \cdots < r_k$,
Csikvari recursively constructed a family of integral trees 
    $$ T_k(r_1,\ldots, r_k) $$ 
having arbitrarily even diameter. As a consequence of Theorem 2.9 and the proof of Theorem 1.1 in \cite{Csi}, the eigenvalues of 
$ T_k(r_1,\ldots, r_k) $ 
are given by 
    $$\{ \pm\sqrt{r_1}, \ldots, \pm \sqrt{r_k}, 0\},$$ 
and hence, for $r_1, \ldots, r_k$ even squares, $T_k(r_1,\ldots, r_k)$ is an integral tree with all even eigenvalues. 
\end{rem}

\begin{rem}
It is also possible, by using some different graph operations, to obtain even or odd graphs starting from another integral or even/odd graphs. 
\begin{enumerate}[($i$)] 
\item The Kronecker product of $K_{2n}$'s is an odd graph (the Crown graph $Cr(4k)$ is a particular case).

\item The balanced blowup $G^{(2m)}$ of any integral graph $G$ is an even graph (in general, the the eigenvalues of the blow-up $G^{(t)}$ are multiples of $t$).

\item The cartesian product $\Box_{i=1}^m K_{2n}$ of complete graphs is even or odd depending whether $m$ is even or odd. 

\item If $G$ es $k$-regular of even order $n=2m$ with even/odd spectrum, then the complementary graph $G^c$ is $(n-k-1)$-regular of order $n$ with odd/even spectrum. 

\item If $G_1,\ldots,G_s$ are even graphs, then NEPS$(G_1,...,G_s; \mathcal{B})$ has even spectrum for every basis $\mathcal{B}$, that is, for every NEPS.
\end{enumerate}
\end{rem}

\begin{rem}
As can be seen from the examples in the previous remark, integral odd spectrum seems more difficult to find than integral even spectrum. 
Also, some of the graphs in the examples are Cayley graphs but some others not. In fact, it is clear that 
    $$K_n=X(\Z_n, \Z_n^\times), \quad K_{n,n}=X(\Z_n \times \Z_2, \Z_n \times \{1\}) \quad \text{and} \quad Q_n = X(\Z_2^n, \{e_i\}_{i=1}^n),$$ 
where $e_i$ is the canonical vector with a $1$ in position $i$. However, trees are not Cayley graphs because they are not regular (with the trivial exceptions of the single vertex graph $0_1$ and the 2-path $P_2$). 

Our construction using mirror di-Cayley (sum) graphs provides, for each pair $(G,S)$ with $e\notin S$ such that $X(G,S)$ is integral, one even graph, $MX^*(G;S,S)$, and one odd graph, $MX^*(G;S,S \cup \{e\})$. 
In addition, by ($d$) in Proposition \ref{prop: directedness}, they are Cayley (sum) graphs: 
\begin{gather*}
MX^*(G;S,S) = X^*(G\times \Z_2, S\times \Z_2), \\[.5mm] 
MX^*(G;S,S\cup \{e\}) = X^* (G \times \Z_2, S \times \Z_2 \cup \{e \times 1\}).
\end{gather*}
\end{rem}

\subsection*{Even and odd di-Cayley (sum) graphs}
We first give general conditions for the existence of even and odd mirror di-Cayley (sum) graphs. 
We will need some definitions and notation. 

Let $X$ be a set and $\mathcal{S}$ a family of subsets of $X$, that is $\mathcal{S} \subset \mathcal{P}(X)$. The \textit{Boolean algebra} $\mathcal{B}(\mathcal{S})$ of $\mathcal{S}$ is the set of all subsets of $X$ that can be obtained by taking a finite number of unions, intersections and complements of members of $\mathcal{S}$ in an arbitrary way. 
A subset $S$ of a group $G$ is called \textit{power-closed} (see \cite{GS}) or \textit{Eulerian} (see \cite{Guo}) if for every $x\in S$ and $y\in \langle x \rangle$ such that $\langle y \rangle = \langle x \rangle$ we have that $y\in S$.

\begin{prop} \label{prop: integral MDCGs}
Let $G$ be a finite group, $S$ a subset of $G$ with $e\notin S$ and let $T \in \mathcal{S}$ with $\mathcal{S} = \{\{e\}, S,S\cup \{e\}\}$. 
Then, we have:
\begin{enumerate}[$(a)$]
    \item If $G=\Z_n$ is cyclic, then $MX^*(G;S,T)$ is integral if and only if $S=\bigcup_{d\in D} S_n(d)$ where $S_d(n) = \{  a \in \Z_n: \gcd(a,n)=d\}$ and $D \subset \{ 1< d <n : d\mid n \}$. \sk 

    \item If $G$ is abelian, then $MX^*(G;S,T)$ is integral if and only if $S \in \mathcal{B}(\mathcal{S}(G))$, where $\mathcal{S}(G)$ denotes the set of subgroups of $G$. \sk 

    \item If $S$ is normal, then $MX^*(G;S,T)$ is integral if and only if $S$ is Eulerian.
\end{enumerate}
\end{prop}

\begin{proof}
For ($a$), Wasin So proved in \cite[Theorem 7.1]{So} that a circulant graph $X(\Z_n,S)$ is integral if and only if $S$ is the union of sets of the form $S_n(d)$ for a set of proper divisors $d\mid n$. \sk 

For ($b$), Klotz and Sander proved in \cite[Theorems 2 and 4]{KS}
that an abelian Cayley graph $X(G,S)$ is integral if and only if $S$ belongs to the Boolean algebra generated by the subgroups of $G$. \sk 

For ($c$), Godsil and Spiga proved in \cite[Theorem 1.1]{GS} (see also \cite[Theorem]{Guo} and \cite[Theorem 1]{KL}) that a normal Cayley graph $X(G,S)$ is integral if and only if $S$ is Eulerian. 

The statement follows by combining all these results with Corollary \ref{coro: integral}. 
\end{proof}

Let us now see some explicit examples of integral Cayley graphs defined over cyclic, abelian and non-abelian groups, respectively.

\begin{exam}[\textit{Integral circulant Cayley graphs: gcd-graphs}] \label{exam: circulant} \
The integral circulant graphs $X(\Z_n, S)$ from ($a$) in Proposition \ref{prop: integral MDCGs} are exactly the \textit{gcd-graphs} $X(\Z_n, D)$ studied by Klotz and Sander in \cite{KS07}, where $D$ is a set of proper divisors of $n$, that is $D \subseteq \{ 1\le d < n : d \mid n \}$. They showed that the eigenvalues of $X(\Z_n,D)$ are 
    $$ \lambda_r = \sum_{d \in D} c(r, \tfrac nd), \qquad (0\le r\le n-1) $$
where 
    $$ c(r, n)= \sum_{\substack{1\le j \le n-1 \\  (j,n)=1}} \omega^{rj} \in \Z $$
with $\omega = e^{\frac{2 \pi i}{n}}$ are the Ramanujan sums, which are known to be integral.  
\hfill $\diamond$
\end{exam} 

\begin{exam}[\textit{Integral abelian Cayley graphs: Sudoku graphs}] \label{exam: sudoku}  \
An $n$-Sudoku, with $n \ge 2$, is an arrangement of $n \times n$ square blocks each consisting of $n \times n$ cells, with each cell filled with a number from $\{ 1,2, \ldots,n^2 \}$ such that every block, row or column contains all of the numbers $1,2, \ldots,n^2$. 

The \textit{Sudoku graph} $Sud(n)$ is defined in \cite{KS10} and \cite{Sa} to have vertices the $n^4$ cells of an $n$-Sudoku such that distinct  vertices  (cells)  are  adjacent if and only if they are in the same block, row or column (they also defined the \textit{positional Sudoku graph} $SudP(n)$). 
It  can  be  shown  that both $Sud(n)$ and $SudP(n)$ are Cayley graphs defined over the abelian group $\Z_4^n$ (see  Section 4.2 in \cite{KS10} for  an  explicit  expression  of  the  corresponding  connection  sets). In \cite{Sa} it is  proved that $Sud(n)$ and $SudP(n)$ are integral, where the spectrum of $Sud(n)$ is given by 
	$$\left( \begin{matrix}
		3n^2-2n-1 & 2n^2-2n-1 & n^2-n-1 & n^2-2n-1 & -1 & -1-n \\ 
		1 & 2(n-1) & 2n(n-1) & (n-1)^2 & n^2(n-1)^2 & 2n(n-1)^2    
		\end{matrix} \right) $$     
and the spectrum of $SudP(n)$ is given by 
$$ \left( \begin{matrix}
4n(n-1) & 2n^2-3n & n(n-2) &0 &-n &-2n \\ 
1 &4(n-1) &4(n-1)^2 &(n-1)^4 &4(n-1)^3 & 2(n-1)^2    
\end{matrix} \right). $$
Here, the first rows of the arrays stand for the eigenvalues and the second rows for the multiplicities.

Also in \cite{KS10} the \textit{pandiagonal Latin square graphs} $PLSG(n)$ are defined. The authors showed that $PLSG(n)$ is a Cayley graph over $\Z_n^2$ with integral spectrum.
\hfill $\diamond$ \end{exam}

\begin{exam}[\textit{Integral non-abelian Cayley graphs: $\mathbb{S}_n$ and $\mathbb{A}_n$}] \label{exam: Sn}  \

\noindent ($a$) 
Let $\mathbb{S}_n$ the symmetric group in $n$ letters and consider the subgroup of all the transpositions 
    $$ T=\{(12), \ldots, (1n)\}.$$  
The star graph 
    $X(\mathbb{S}_n, T)$ 
has integral spectrum with eigenvalues 
    $$\{n, n-1, \ldots, 1,0,-1, \ldots, -n+1, -n\}$$ 
(see Theorem \cite{KraMo}, the multiplicities are computed in Corollary 2.1 in \cite{ChaFe}). \sk 

\noindent ($b$)
The derangement graph $X(\mathbb{S}_n, \mathcal{D}_n)$, where 
    $$ \mathcal{D}_n = \{ \sigma \in \mathbb{S}_n : \sigma(i) \ne i, \, 1\le i \le n\}$$ 
is the set of $n$-derangements, is integral with eigenvalues given by (see Theorem 3.2 in \cite{Re})
    $$ \eta_\lambda = \sum_{k=0}^n (-1)^k \frac{n!}{(n-k)!} \frac{f_{\lambda/k}}{f_\lambda}$$ 
where $\lambda$ runs over the sets of all integer partitions of $n$ and $f_\lambda$ (resp.\@ $f_{\lambda/\mu}$) is the number of standard (resp.\@ semistandard) Young tableaux of shape $\lambda$ (resp.\@ $\lambda/\mu$). \sk

\noindent ($c$) 
There are also families of integral Cayley graphs $X(\mathbb{A}_n, S)$ over the alternating group in $n$ letters $\mathbb{A}_n$.
For instance, for any $n \ge 4$ and $1\le k \le n$, the Cayley graph    
$$ X(\mathbb{A}_n, T_k) \quad \text{where} \quad T_k = \{ (k,i,j) \in \mathbb{A}_n\} $$ 
is integral with eigenvalues (see Theorem 3 in \cite{Guo}) 
    $$ \{ i^2 -n+1: 1\le i \le n\}.$$
Also, the Cayley graph 
    $$ X(\mathbb{A}_n, S) \quad \text{with} \quad S = \{ (1,2,i), (1,2,i)^{-1} : 3 \le i \le n\}$$ 
is integral for any $n\ge 3$ (see Corollary 5 in \cite{KL}).

See Section 3.3 in \cite{LZ} for more details on these examples.
\hfill $\diamond$ \end{exam}

\begin{exam}[\textit{Integral non-abelian Cayley graphs: $\mathbb{D}_n$, $S\mathbb{D}_n$}, and $Dic_n$] \label{exam: Dn}  \
The dihedrant graph is the Cayley graph $X(\mathbb{D}_n, S)$ defined over the dihedral group 
    $$ \mathbb{D}_n = \langle a,b : a^n= b^2=1, b^{-1}ab = a^{-1} \rangle $$ 
of order $2n$ with connection set $S=\mathbb{D}_n \smallsetminus \{e\}$. 
There are some results on the integrality of dihedrants and  Cayley graphs $X(S\mathbb{D}_n, S)$ over the semidihedral group    $$S \mathbb{D}_n = \langle a,b : a^{4n}= b^2=1, bab=a^{2n-1} \rangle = \langle a \rangle \cup b\langle a \rangle$$ 
of order $8n$, $n\ge 2$, for a normal connection set $S$ (see Theorems 51-58 in Section~3.4 in \cite{LZ}), some of them in terms of the sets $S_d(n)$ as in ($a$) of Proposition \ref{prop: integral MDCGs}. 
However, the spectra of the graphs are not given explicitly in these cases.  

On the other hand, the Cayley graphs $X(Dic_n, S)$ defined over the dicyclic group 
    $$ Dic_n = \langle a,b : a^{2n}=1, a^n=b^2, b^{-1}ab = a^{-1} \rangle$$
of order $4n \ge 8$, with connection set $S=\{ a^k : 1\le k \le 2n-1, k\ne n\} \cup \{ab, a^{n+1}b\}$ is integral for any odd $n$ with even spectrum given by (see Theorem 1.3 in \cite{AV}) 
  $$ Spec(X(Dic_n, S)) = \{ [2n]^1, [2n-4]^1, [0]^{3n-1}, [-4]^{n-1} \} .$$
See Theorems 59--60 and 62--64 in Section 4 of \cite{LZ} for more integrality results on Cayley graphs defined over dicyclic groups.
\hfill $\diamond$ \end{exam}

Putting together the results of the section we get an existence result for even and odd MDCGs.
\begin{thm} \label{teo: cayley even/odd}
We have the following.
\begin{enumerate}[$(a)$]
    \item There are even MDCGs and odd MDCGs defined either over cyclic, abelian or non-abelian groups. \sk 

    \item There are even Cayley (sum) graphs and odd Cayley (sum) graphs defined over non-cyclic abelian groups and also over non-abelian groups.
\end{enumerate}
In the non-abelian case, the group can be chosen to be  $\mathbb{S}_n$, $\mathbb{A}_n$, $\mathbb{D}_n$,  $S\mathbb{D}_n$ or $Dic_n$.
\end{thm}

\begin{proof}
For ($a$), by Examples \ref{exam: circulant}--\ref{exam: Dn}, there are integral Cayley graphs $X(G,S)$ with $G$ either cyclic, abelian or non-abelian. 
By ($a$) in Corollary \ref{coro: integral}, the mirror di-Cayley (sum) graphs $MX^*(G;S,T)$ with $T=\{e\}, S$ or $S\cup \{e\}$ are all integral. Moreover, by ($b$) in Corollary \ref{coro: integral}, 
$MX^*(G;S,S)$ is an even MDCG and $MX^*(G;S,S \cup \{e\})$ is an odd MDCG.
Finally, for ($b$), $MX^*(G;S,S)$ and $MX^*(G;S,S \cup \{e\})$ are Cayley graphs  defined over $G \times \Z_2$, by Proposition \ref{prop: cayley structure}.
\end{proof}

\subsection*{Circulant even or odd graphs}
We finish the section with a natural inquiry.
Notice that by the previous theorem it is possible to obtain Cayley graphs having even and odd spectrum over $G\times \Z_2$, with $G$ cyclic, abelian or non-abelian. So, our construction does not allow us to produce even/odd circulant graphs (i.e., over cyclic groups), if they exist.

The first question is: \textit{are there even or odd circulant graphs?} 
The answer is yes. 
A quick search using Python for Cayley (sum) graphs with $G$ of order up to 16 shows that there are such even and odd circulant graphs for all the orders considered.
A more detailed search shows that there are some general families like: 
\begin{enumerate}[$(a)$]
\item For each $n \ge 1$, let $S_{odd} = \{1,3,5,\ldots, 4n-1\}$. The graphs 
	$$\G_{4n, \, ev}^* = X(\Z_{4n}, S_{odd}) \qquad \text{and} \qquad \G_{4n, \, odd}^* = X(\Z_{4n}, S_{odd} \cup \{0\})$$ 
have integral spectrum, respectively even and odd, given by
 	$$\{[2n]^1, [0]^{4n-2}, [-2n]^1\} \qquad \text{and} \qquad \{[2n+1]^1, [1]^{4n-2}, [-2n+1]^1\}.$$ 
The graphs $\G_{4n, \, ev}$ coincide with its sum version $\G_{4n, \, ev}^+$ and are loopless, connected, bipartite and strongly regular graphs (SRG for short).
The graphs $\G_{4n, \, odd}$ are looped and connected while $\G_{4n, \, odd}^+$ are loopless and connected.

\msk

\item For each $n \ge 3$, let $S = \Z_{4n} \smallsetminus \{0,n,2n,3n\}$. The graphs 
	$$\tilde \G_{4n, \, ev}^* = X(\Z_{4n}, S) \qquad \text{and} \qquad \tilde \G_{4n, \, odd}^* = X(\Z_{4n}, S \cup \{0\})$$ 
have integral spectrum, respectively even and odd, given by
	$$ \{ [4n-4]^1, [0]^{3n}, [-4]^{n-1}\} \qquad \text{and} \qquad \{[4n-3]^1, [1]^{3n}, [-3]^{n-1}\}.$$ 
The graphs $\tilde \G_{4n, \, ev}$ coincide with its sum version $\tilde \G_{4n, \, ev}^+$ and are loopless, connected, bipartite and SRG. The graphs $\tilde \G_{4n, \, odd}$ are looped and connected while $\tilde \G_{4n, \, odd}^+$ are loopless and connected.	
\end{enumerate}

However, these are well-known families of strongly regular graphs with parameters $srg(n,k,e,d)$. For instance, the graphs $\G_{4n,ev}^*$ are strongly regular with parameters
	\[ srg(4n,\,2n,\,0,\,2n), \]
and are isomorphic to the complete bipartite graph $K_{2n,2n}$ (see $(iii)$ of Remark \ref{rem: even odd examples}), while 
the graphs $\tilde{\G}_{4n,ev}^*$ are strongly regular with parameters
	\[ srg(4n,\,4n-4,\,4n-8,\,4n-4), \]
and are isomorphic to the complete $n$-partite graph $K_{4,4,\ldots,4}$ ($n$ times), also denoted $K_{4 \times n}$, where each part has exactly 4 vertices 
(alternatively, the complement of $n$ disjoint copies of $C_4$).

Furthermore, in 1981, Bridges and Mena (\cite{BM}, see also \cite{Ma}) characterized all circulant graphs which are SRG.
In fact, they showed that besides the trivial graphs (i.e., the complete graph $K_n=Cay(\Z_n,\Z_n \smallsetminus \{0\})$ and the empty graph $0_n=Cay(\Z_n, \varnothing)$), there are only three families of circulant strongly regular graphs:
the imprimitive families (either the graph or the complement are disconnected) and the classical Paley graphs of prime order $p$ with $p\equiv 1 \pmod 4$.

The Paley graph $P(p)=Cay(\Z_p, R_p)$, with $R_p=\{x^2 : x\in \Z_p\}$, has parameters $srg(p,\frac{p-1}2,\frac{p-5}4, \frac{p-1}4)$ are not integral since they have spectrum $\{ \frac{p-1}2, \frac{-1-\sqrt p}2, \frac{-1+\sqrt p}2\}$. 

The imprimitive families are the complete multipartite graphs $K_{r \times m}$ and their non-connected complements $mK_r$.
The complete multipartite graphs $K_{r, r, \dots, r}$ (with $m$ parts) has parameters 
	$$ srg(mr, (m-1)r, (m-2)r, (m-1)r)$$ 
and spectrum 
	$$ \{[r(m-1)]^1, [0]^{m(r-1)}, [-r]^{m-1}\}. $$
On the other hand, the complement of $K_{ \times m}$ are the disjoint union of complete graphs $m K_r$ with parameters
$srg(mr, r-1, r-2, 0)$.

Thus, the only \textit{integral connected circulant strongly regular graphs} are given by complete multipartite graphs $K_{r \times m}$, which for certain parameters are (or can be adapted to be) even or odd. In fact, the families that we found in ($a$) and ($b$) above are of this form: if $m=2$ and $r=2n$ we get $K_{2 \times 2n}=K_{2n,2n}$ which are the graphs in ($a$) and if $r=4$ we get $K_{4 \times m} = K_{4,4,\ldots,4}$ ($m$ times) which are the graphs in ($b$).

Hence, all these comments lead us to the following:
\begin{openquest}
\textit{Is there a general way to construct (non-trivial, connected) circulant Cayley graphs $X(\Z_n, S)$ and $X^+(\Z_n, S)$ having even and odd spectrum, which are not strongly regular?}
\end{openquest}

\section{Even/odd Cayley graphs over finite commutative rings} \label{sec: integral GR}

Here we give some general examples of integral MDCGs in $\mathcal{F}$ having either even or odd  spectrum, constructed from Cayley graphs over finite commutative rings. We will give two different constructions: one using $k$-th powers in finite fields, and the other using units in finite commutative rings.

\subsection{Even/odd MDCGs from GP-graphs over finite fields}
Let $\ff_q$ be a finite field of $q$ elements and let $k$ be a positive integer such that $k \mid q-1$. Taking $G=\ff_q$ and $P_k = \{x^k : x \in \ff_q^*\}$ we have the \textit{generalized Paley graph} (GP-graphs for short) $$\G(k,q) = X(G,P_k)$$ and the corresponding GP sum graph 
	$\G^+(k,q)$ (for general information on these graphs see the works of Nguyen and coauthors and Podestá-Videla). 
	
One can consider mirror di-Cayley graphs over $\ff_q$ in the family $\mathcal{F}$ generalizing in some way the GP-graphs. Namely, for $k \mid q-1$, we have the three mirror di-Cayley (sum) graphs 
    $$ MX^*(\ff_q; P_k, \{0\}), \qquad MX^*(\ff_q; P_k, P_k), 
    \qquad MX^*(\ff_q; P_k, P_k  \cup \{0\}).$$	
The criterion for integrality (and even/odd parity) of these graphs is very simple.
\begin{prop}
Let $q=p^m$ for some $m \in \N$ with $p$ prime and let $k \in \N$ such that $k\mid q-1$. Then, we have that 
$MX^*(\ff_q; P_k, \{0\})$ is integral, $MX^*(\ff_q; P_k, P_k)$ is even and $MX^*(\ff_q; P_k, P_k  \cup \{0\})$ is odd if and only if $k\mid \frac{q-1}{p-1}$.
\end{prop}	

\begin{proof}
It is known, by Theorem~4.1 in \cite{PV2}, that a GP-graph $\G(k,q)$ is integral if and only $k\mid \frac{q-1}{p-1}$ . The result thus follows by Corollary \ref{coro: integral}. 
\end{proof}

Some few families of GP-graphs are known to be integral. To our best knowledge, we have the following ones with known spectrum.

\begin{exam}[\textit{Families of integral GP-graphs with known spectrum}.] \label{exam: GP-enteros}

\

\noindent $(a)$
\textit{$\G(k,q)$ with small $k$, $1\le k \le 4$}. The graph $\G(1,q)$ is the complete graph $K_q$ which is integral. The graph 
$\G(2,q)$ is the classic (undirected) Paley graphs $P(q)$ for $q\equiv 1 \pmod 4$ and the directed Paley graph $\vec{P}(q)$ for $q \equiv 3 \pmod 4$. The spectrum of these graphs are well-known (see for instance Examples 2.3 and 2.4 in \cite{PV2}). The spectrum of $\G(3,q)$ and $\G(4,q)$ are covered in Theorems 3.1 and 3.2 of  \cite{PV2}, respectively, and are more involved. In Example 4.2 of \cite{PV2} one can find for which values of $k$ and $q$ these graphs are integral.

\sk 

\noindent $(b)$
\textit{Hamming GP-graphs}. 
A Hamming graph $H(b, q)$ is a graph with vertex set $V = K^b$ where $K$ is any set of size $q$ (typically $\ff_q$ in applications), and where two $b$-tuples form an edge if and only if they differ in exactly one coordinate. Notice that 
	$$ H(b, q) = \square^b K_q $$ 
and hence, Hamming
graphs are integral with spectrum 
\begin{equation} \label{eq: spec Hamming}
	Spec(H(b, q)) = \{ [\ell q - b]^{\binom{b}{\ell} (q-1)^{b-\ell}} : 0\le \ell\le b\}.
\end{equation}  
Connected GP-graphs $\G(k, q)$ which are Hamming graphs were classified by Lim and Praeger in \cite{LP}. In this case $k =\frac{p^{bm}-1}{b(p^m-1)}$ where $b | \frac{p^{bm}-1}{p^m-1}$. 
Hence, 
\begin{equation} \label{eq: Hamming}
	\G \Big( \tfrac{p^{bm}-1}{b(p^m-1)},p^{bm} \Big) =    H(b,p^m).    
\end{equation}

\noindent $(c)$
\textit{Semiprimitive GP-graphs}. 
A GP-graph $\G(k,q)$ is semiprimitive if $k=2$ and $q\equiv 1 \pmod 4$ or else 
$k \ge 3$ and 
$$k\mid p^{t}+1 \qquad  \text{for some $t \mid \tfrac m2$ with $t\ne \tfrac m2$}.$$ 
In Proposition~5.3 of \cite{PV2} it is proved that every semiprimitive GP-graph is integral. In Theorem 5.4 of \cite{PV2} the authors explicitly give the spectrum of $\G(k,q)$ and $\G^+(k,q)$. 

In fact, let $(k, q)$ be a semiprimitive pair with $q=p^m$, $m$ even, and $n=\frac{q-1}k$.  
Then, the spectra of $\G=\G(k,q)$ and $\G^+=\G^+(k,q)$ 
are integral. We have 
\begin{equation} \label{eq: spec semip}
	\mathrm{Spec}(\G) 		 = \big\{ [n]^1, [\lambda_1]^n, [\lambda_2]^{(k-1)n} \big\}, 
\end{equation}
with 
$$	\lambda_1 = \frac{(-1)^{\frac{m}{2t}+1}(k-1) p^{\frac m2}-1}{k}
\qquad  \text{and} \qquad  
\lambda_2 = -\frac{(-1)^{\frac{m}{2t}+1} p^{\frac m2}+1}{k}, $$ 
where $t$ is the least integer $j$ such that $k \mid p^j+1$ (hence $\frac{m}{2t} \ge 1$). 
Furthermore, 
$\mathrm{Spec}(\G^+) = \mathrm{Spec}(\G)$ if $q$ is even and $$\mathrm{Spec}(\G^+) = \big\{ [n]^1, [\pm \lambda_1]^{\frac n2}, [\pm \lambda_2]^{\frac{(k-1)n}2} \big\}$$ if $q$ is odd. 
\hfill $\diamond$	
\end{exam}

\begin{rem}
Using Proposition \ref{prop: spec bicayleys}, it is easy to obtain explicit expressions for the spectra of the mirror di-Cayley (sum) graphs $MX^*(\ff_q; P_k,T)$ with $T=\{0\}, P_k$ or $P_k\cup \{0\}$ associated to the GP-graphs with $k$ small ($1\le k\le 4$), for Hamming GP-graphs and for semiprimitive GP-graphs, from the spectrum of these GP-graphs in items $(a)$--$(c)$ of the above example.
We leave the details of this to the interested reader.
\end{rem}

Very recently, using the Euler function $\varphi$ and cyclotomic polynomials $\Phi_d(x)$, some infinite families of integral GP-graphs were obtained in \cite{PV3}, although the spectrum is not explicitly known (of course, it is given by Gaussian periods in general, see Theorem 2.1 in \cite{PV2}). We give here the main result borrowed from \cite{PV3}; many more concrete families and examples can be seen in Section 5 of \cite{PV3}. 

\goodbreak 

\begin{prop}{$($\cite[Theorem 5.16]{PV3}$)$} \label{thm: general integral GP-graphs}
	Let $p$ be a prime and $k \in \N$. We have the general infinite families of integral GP-graphs:
	\begin{enumerate}[$(a)$]
		\item $\{\G(k \frac{q^{ta}-1}{q^t-1}, q^{at})\}_{a,t\in \mathbb{N}}$ where $q=p^k$, provided that $k\mid p-1$. \msk
		
		\item $\{\G(k \frac{q^{ta}-1}{q^t-1}, q^{at})\}_{a,t\in \mathbb{N}}$ where $q=p^2$, provided that $k\mid p+1$. \msk
		
		\item $\{\G(k \frac{q^{ta}-1}{q^t-1}, q^{at})\}_{a, t\in \mathbb{N}}$ 
		where $q=p^{\varphi(k)}$, provided that $\gcd(k,p(p-1))=1$. \msk 
		
		\item $\{\G(\Phi_d(p) \frac{q^{ta}-1}{q^t-1}, q^{at}) \}_{a,t\in \N}$ with $q=p^d$ and $d>1$. 
	\end{enumerate}
\end{prop}

We now give, in two theorems, general explicit families of even and odd (mirror di-Cayley (sum)) graphs defined over finite fields.

\begin{thm} \label{thm: even/odd over Fqs}
Let $\ff_q$ be a finite field of $q=p^m$ elements with $p$ prime. Let $k$ be a divisor of $q-1$ and take $P_k=\{ x^k: x\in \ff_q^*\}$. Then we have:
\begin{enumerate}[$(a)$]
    \item The graph $MX^*(\ff_q; \ff_q^*, \ff_q^*)$ is even and the graph $MX^*(\ff_q; \ff_q^*, \ff_q)$ is odd. \msk 

    \item The graphs $MX^*(\ff_q; P_2, P_2)$ and $MX^*(\ff_q; P_2, P_2 \cup \{0\})$ are respectively even and odd if and only if $q\equiv 1 \pmod 4$. \msk 
    
    \item The graphs $MX^*(\ff_q; P_3, P_3)$ and $MX^*(\ff_q; P_3, P_3 \cup \{0\})$ are respectively even and odd if and only if $p\equiv 1 \!\pmod 3$ with $3\mid m$ or $p\equiv 2 \pmod 3$ and $m$ even.  \msk 
    
    \item The graphs $MX^*(\ff_q; P_4, P_4)$ and $MX^*(\ff_q; P_4, P_4 \cup \{0\})$ are respectively even and odd if and only if $p\equiv 1 \pmod 4$ with $m\equiv 0 \pmod 4$ or $p\equiv 3 \pmod 4$. \msk 

    \item The graph $MX^*(\ff_{p^{bt}}; P_k, P_k)$ is even and the graph $MX^*(\ff_{p^{bt}}; P_k, P_k \cup \{0\})$ is odd for any $k=(p^{bt}-1)/b(p^t-1)$ with $b$ a divisor of $(p^{bt}-1)/(p^t-1)$. \msk 

    \item The graph $MX^*(\ff_{q}; P_k, P_k)$ is even and the graph $MX^*(\ff_{p}; P_k, P_k \cup \{0\})$ is odd for any semiprimitive pair of integers $(k,q)$.
\end{enumerate}
\end{thm}

\begin{proof}
By part ($ii$) in ($b$) of Corollary \ref{coro: integral} and Remark \ref{rem: X int X*}, it is enough to show that the Cayley graph $X(\ff_q, P_k)$ associated to every MDCG in each item is integral. 

Item ($a$) is clear.
That $X(\ff_q, P_k)$  is integral for $k=2,3,4$ if and only if $k$ and $q$ are as stated in items ($b$)--($d$) follows by Example 4.2 in \cite{PV2}. That the Cayley graphs $X(\ff_q, P_k)$ in $(e)$ and $(d)$ are integral follows from \eqref{eq: spec Hamming}, \eqref{eq: Hamming},  and \eqref{eq: spec semip} in Example~\ref{exam: GP-enteros}.
\end{proof}

We now give infinite families of integral (mirror di-Cayley (sum)) graphs.

\begin{thm} \label{thm: infinite families of even/odd graphs over Fqs}
Let $q=p^m$ with $p$ a prime and $k,a,t \in \N$ and put $k'=k \frac{q^{at}-1}{q^t-1}$. 
In the previous notations, we have the families 
$$ \{ MX^*(\ff_q; P_{k'},P_{k'}) \}_{a,t \in \N}  \qquad \text{and} \qquad \{ MX^*(\ff_q; P_{k'},P_{k'}\cup \{0\}) \}_{a,t \in \N}$$
of even and odd MDCGs respectively, in the following situations:
\begin{enumerate}[$(a)$]
\item $m=k$ and $k\mid p-1$ or $m=2$ and $k\mid p+1$. \msk 

\item $m=\varphi(k)$ and $\gcd(k,p(p-1))=1$. \msk 

\item $m=d>1$ and $k=\Phi_d(p)$.  
\end{enumerate}
\end{thm}

\begin{proof}
This is a direct consequence of Proposition \ref{thm: general integral GP-graphs} together 
with Corollary \ref{coro: integral}.
\end{proof}

\subsection{Even/odd MDCGs from unitary Cayley graphs over rings}
 Here we show that any MDCG in $\mathcal{F}$ defined over a finite commutative ring $R$ with connection set its groups of units $R^*$ is integral (hence $0 \notin R^*$). 
That is, 
    $$ Spec(MX^*(R;R^*, T)) \subset \Z$$ 
for any $T \in  \mathcal{R}$, where $\mathcal{R}=\{ \{0\},R^*, R^* \cup \{0\}\}$.

\begin{prop} \label{prop: GRR integral}
Let $R$ be a finite commutative ring with identity and let $R^*$ be its group of units. Then, the mirror di-Cayley (sum) graphs 
\begin{align*}
G_{R,0}^*         & := MX^*(R;R^*, \{0\}), \\ 
G_{R,R^*}^*       & := MX^*(R;R^*, R^*), \\ 
G_{R,R^*\cup \{0\}}^* & := MX^*(R;R^*, R^* \cup \{0\}),
\end{align*}
are integral. Moreover, $G_{R,R^*}^*$ is even and $G_{R,R^* \cup \{0\}}^*$ is odd.
\end{prop}

\begin{proof}
It is known that $G_R = X(R,R^*)$   
has integral spectrum (see \cite{ABetal}, \cite{KAYS}) and, hence, $G_R^+=X^+(R,R^*)$ also has integral spectrum by Lemma \ref{lem: eigenvalues X*GS}. Thus, the graphs $G_{R,0}^*$, 
$G_{R,R^*}^*$ 
and $G_{R,R^*\cup \{0\}}^*$ 
have integral spectrum, with $G_{R,R^*}^*$ even and $G_{R,R^*\cup \{0\}}^*$ odd, by Corollary \ref{coro: integral}.   
\end{proof}

Notice that, for finite local rings $R$, the graph $G_R$ is always undirected and loopless. Also, if
$R$ is of even size we have $G_R=G_R^+$, see Lemma~3.1 in \cite{PV}. 
For finite local rings $R$ of odd size $G_R^+$ has loops and $G_R \ne G_R^+$. It is easy to give the spectra of $G^*_{R,0}$, $G^*_{R,R^*}$ and $G^*_{R,R^* \cup \{0\}}$ in this case, explicitly. 
\begin{coro} \label{coro: SpecGRR*}
Let $(R,\frak{m})$ be a finite local ring of odd size $|R|=r$, with maximal ideal $ \frak{m}$ of size $|\frak{m}|=m$ (if $m=1$, $R$ is a field). 
Then, the integral spectrum of $G^*_{R,0}$ (which is neither even nor odd) is  given by
\begin{align*}
	Spec(G_{R,0}) & = \begin{cases}
			\big\{ [r-m \pm 1]^1, [\pm 1]^{\frac rm(m-1)},  [-m \pm 1]^{\frac{r-m}{m}} \big\} & \qquad \qquad \qquad \text{if $m \ge 2$,} \\[1mm]
			\big\{ [r]^1, [r-2]^1, [0]^{r-1}, [-2]^{r-1} \big\} & \qquad \qquad \qquad \text{if $m =1$,} \end{cases} 
		\\[2mm]
		Spec(G_{R,0}^+) & = \begin{cases}
			\big\{ [r-m \pm 1]^1, [m\pm 1]^{\frac{r-m}{2m}}, [\pm 1]^{\frac rm(m-1)},  [-m \pm 1]^{\frac{r-m}{2m}} \big\} & \quad \text{if $m \ge 2$,} \\[1mm]
			\big\{ [r]^1, [r-2]^1, [0]^{r-1}, [\pm 2]^{\frac{r-1}2} \big\} & \quad \text{if $m =1$,}
		\end{cases}    
	\end{align*}
the even spectrum of $G^*_{R,R^*}$ is given by  
\begin{align*}
		Spec(G_{R,R^*}) & = \begin{cases}
			\big\{ [2(r-m)]^1, [0]^{2r-\frac rm},  [-2m]^{\frac{r-m}{m}} \big\} & \qquad \text{if $m \ge 2$,} \\[1.5mm]
			\big\{ [2(r-1)]^1, [-2]^{r-1}, [0]^{r} \big\} & \qquad \text{if $m =1$.} \end{cases} 
		\\[2mm]
		Spec(G_{R,R^*}^+) & = \begin{cases}
			\big\{ [2(r-m)]^1, [0]^{2r-\frac rm},  [\pm 2m]^{\frac{r-m}{2m}}  \big\} & \qquad \text{if $m \ge 2$,} \\[1.5mm]
			\big\{ [2(r-1)]^1, [0]^{r}, [\pm 2]^{\frac{r-1}2} \big\} & \qquad \text{if $m =1$,}
		\end{cases}      
\end{align*}
and the odd spectrum of $G^*_{R,R^*\cup 0}$ is given by
\begin{align*}
Spec(G_{R,R^*\cup \{0\}}) & = \begin{cases}
	\big\{ [2(r-m)+1]^1, [1]^{2r-\frac rm},  [-2m+1]^{\frac{r-m}{m}} \big\} & \qquad \text{if $m \ge 2$,} \\[1mm]
	\big\{ [2r-1]^1, [-1]^{r-1}, [1]^{r} \big\} & \qquad \text{if $m =1$.} \end{cases} 
		\\[2mm]
		Spec(G_{R,R^*\cup \{0\}}^+) & = \begin{cases}
	\big\{ [2(r-m)+1]^1, [1]^{2r-\frac rm},  [\pm 2m+1]^{\frac{r-m}{2m}}  \big\} & \qquad \text{if $m \ge 2$,} \\[1mm]
			\big\{ [2r-1]^1, [1]^{r}, [3]^{\frac{r-1}2}, [-1]^{\frac{r-1}2} \big\} & \qquad \text{if $m =1$.}
		\end{cases}      
	\end{align*}
\end{coro}

\begin{proof}
The result follows after some direct computations, from 
the following expressions for the spectrum of $G_R$ and $G^+_R$ 
(see (3.7) in Proposition 3.2, and (3.9) in its proof, from \cite{PV}): 
\begin{equation} \label{eq: Spec GR local}
    Spec(G_R) = \begin{cases}
        \{ [r-m]^1, [0]^{\frac{r}{m}(m-1)}, [-m]^{\frac rm -1} \} & \qquad m \ge 2, \\[1.5mm]
        \{ [r-1]^1, [-1]^{r-1} \} & \qquad m=1,
    \end{cases}
\end{equation} 
and 
\begin{equation} \label{eq: Spec GR+ local}
    Spec(G_R^+) = \begin{cases}
        \{ [r-m]^1, [m]^{\frac{r-m}{2m}}, [0]^{\frac{r}{m}(m-1)}, [-m]^{\frac{r-m}{2m}} \} & \qquad m \ge 2, \\[1.5mm]
        \{ [r-1]^1, [1]^{\frac{r-1}2}, [-1]^{\frac{r-1}2} \} & \qquad m=1,
    \end{cases}
\end{equation} 
together with Proposition \ref{prop: spec bicayleys}. 
We note for later use that the expression \eqref{eq: Spec GR local} is also valid for $r$ even. 
That the spectrum of $G_{R,0}^*$ is neither even nor odd follows since $m$ has the same parity as $r$, as $r$ is odd by assumption.
\end{proof}

As a nice example, let us consider Galois rings.
\begin{exam}[\textit{Galois rings}]
Let $p$ be an odd prime, $s, t \in \N$, and take 
    $$ R = GR(p^s, t) = \Z_{p^s}[x]/(f_t(x))$$ 
be the finite Galois ring of $p^{st}$ elements, where $f_t(x)$ is a monic irreducible polynomial of degree $t$. The ring $R$ is local with maximal ideal $\frak{m} = (p)$, hence $m = p^{(s-1)t}$ and 
    $$ |R^*| = p^{(s-1)t}(p^t - 1). $$ 
In particular, if $t = 1$ then $R=GR(p^s,1)$ is the local ring $\Z_{p^s}$ while if $s = 1$ we have that $R=GR(p,t)$ is the finite field $\ff_{p^t}$. 
Putting the values 
    $$ r=p^{st} \qquad \text{and} \qquad m=p^{(s-1)t} $$
in the expressions of Corollary \ref{coro: SpecGRR*} we obtain all the spectra of the six graphs $G^*_{R,0}$, $G^*_{R,R^*}$ and $G^*_{R,R^* \cup \{0\}}$, for $R=GR(p^s,t)$.  
\hfill $\diamond$ 
\end{exam}

\black

\section{Isospectrality in $\mathcal{F}$} \label{sec: crossed isosp}
Here we study the isospectrality between certain pairs of MDCGs in $\mathcal{F}$, focusing on three cases. In the first two, the pair $(G,S)$ will be fixed.
We begin by considering the graphs $MX^*(G;S,T)$ and $MX^*(G;S,T')$ for two different subsets $T,T'$. We call this crossed isospectrality. 
Then, we compare the spectrum of a graph $MX(G;S,T)$ with the spectrum of its sum version $MX^+(G;S,T)$.
Finally, we consider the general case of two triples $(G;S,T)$ and $(G';S',T')$.

\subsection{There are no crossed isospectrality in $\mathcal{F}$} \label{subsec: crossed isosp}
Here we show that, as one can probably expect, there is no 
isospectrality between the graphs $MX^*(G;S,T)$ and $MX^*(G;S,T')$ when the pair 
$(G,S)$ is fixed and the subsets $T$, $T'$ are either $\{e\}$, $S$ or 
$S \cup \{e\}$, but different.                           

This is clear when $X(G,S)$ is integral since, in this case, all associated MDCGs $MX^*(G;S,T)$ are integral, but $MX^*(G;S,S)$ is even, $MX^*(G;S,S \cup \{e\})$ is odd, and $MX^*(G;S,T)$ is (generically) neither even nor odd (see Corollary \ref{coro: integral}). This is also true in general.

\goodbreak 

\begin{prop}\label{prop: isospec T,T'}
Let $G$ be a group and $S \subset G$ with $|S|\ge 2$. Then:
\begin{enumerate}[$(a)$]
	\item The graphs $MX^*(G;S,\{e\})$ and $MX^*(G;S,S)$ are not isospectral. \sk 
	
	\item The graphs $MX^*(G;S,\{e\})$ and $MX^*(G;S,S\cup \{e\})$ are not isospectral. \sk 
	
	\item The graphs $MX^*(G;S,S)$ and $MX^*(G;S,S\cup\{e\})$ are not isospectral. 
\end{enumerate}
\end{prop}

\begin{proof}
We will use the abbreviations $\G_e^*=MX^*(G;S,\{e\})$, $\G_S^*=MX^*(G;S,S)$ and $\G_{S \cup \{e\}}^*=MX^*(G;S,S\cup \{e\})$.
Since $|S|\ge 2$, the three graphs $\G_{e}^*$, $\G_{S}^*$ and $\G_{S \cup \{e\}}^*$ 
are all different (see the comments after Proposition \ref{prop: directedness}).

Assume that $Spec(X^*(G,S)) = \{ [\lambda_i]^{m_i^*}\}_{i \in I}$, with $\sum_i m_i = \sum_i m_i^* =|G|$.

\noindent ($a$) 
By Proposition \ref{prop: spec bicayleys}, we know that $\G_e^*$ and $\G_S^*$ 
are isospectral if and only if
$$\{ [\lambda^*_i+1]^{m_i^*}, [\lambda^*_i-1]^{m_i^*} \}_{i\in I} = 
\{ [2\lambda^*_i]^{m_i^*} \}_{i\in I}\cup \{[0]^{|G|}\}.$$

Suppose by contradiction that $\G_e^*$ is isospectral with $\G_S^*$. Hence, $|G|$ eigenvalues of $\G_e^*$ are equal to 0.
First, observe that if $\lambda_i^*+1=0$ for some $i$, then $\lambda_i^*-1\ne 0$ 
(actually, if $\lambda^*_i+1=0$, then $\lambda^*_i-1=2\lambda_i^*$ and, if $\lambda_j^*-1=0$, then $\lambda^*_j+1=2\lambda^*_j$). 
Hence, there exists some $s$ with $0\le s\le |G|$,
such that, after reordering if necessary, $\lambda_i^*+1=0$ for $1\le i \le s$ and $\lambda_j^*-1=0$ for $s+1 \le i \le |G|$. 
That is,  $$ \lambda_1^*+1, \ldots, \lambda_s^*+1, \lambda_{s+1}^*-1, \ldots, \lambda_{|G|}^*-1,$$
are the 0 eigenvalues in $\G_e^*$.
Therefore, $X^*(G,S)$ has as eigenvalues $-1$ with multiplicity $s$ and $1$ with multiplicity $|G|-s$, i.e.
    $$ Spec(X^*(G,S)) = \{[1]^{|G|-s}, [-1]^s\}. $$

If $s=|G|$, we have that $X^*(G,S)$ has $-1$ as its only eigenvalue, which is impossible since the degree of regularity of $X^*(G,S)$ is also an eigenvalue. 

If $s=0$, then the only eigenvalue of $X^*(G,S)$ is $1$. Hence, the Cayley graph is of the form $X^*(G,\{e\})$ for some group $G$, but we have that $|S|\ge 2$; therefore this case cannot occur. 

If $0<s<|G|$, it is well-known that a regular graph with 
only two different eigenvalues is a complete graph. Hence, $X^*(G,S)$ is a complete graph. 
Moreover, since the only positive eigenvalue is $1$, then $X^*(G,S)$ is $1$-regular. Therefore, $X^*(G,S)=P_2$. But, since $|S|\ge 2$, this is impossible.

Thus, $\G_e^*$ and $\G_{S}^*$ are not isospectral.
\sk 

\noindent $(b)$ 
By Proposition \ref{prop: spec bicayleys}, the graphs $\G_e^*$ and 
$\G_{S \cup \{e\}}^*$
are isospectral if and only if
	$$\{[\lambda^*_i + 1]^{m_i^*},[\lambda^*_i-1]^{m_i^*}\}_{i\in I} = \{ [2\lambda^*_i+1]^{m_i^*} \}_{i\in I}\cup\{[-1]^{|G|}\}.$$
Notice that if $\lambda_i+1=-1$, then $\lambda_i-1\ne -1$. Hence, there exists $s$ such that, after reordering if necessary, $\lambda^*_i+1=-1$ for $i=1,\dots,s$ and $\lambda^*_i-1=-1$ for $i=s+1,\dots,|G|$. Therefore, $X^*(G,S)$ has as eigenvalues $-2$ with multiplicity $s$ and $0$ with  multiplicity $|G|-s$. Notice that when $\lambda^*_i+1=-1$ we have that $\lambda^*_i-1=2\lambda^*_i+1$ and when $\lambda^*_i-1=-1$ we have that $\lambda^*_i+1=2\lambda^*_i+1$.
	
Then, $\G_e^*$ and $\G_{S \cup \{e\}}^*$
are isospectral if and only if 
    $$ Spec(X^*(G,S)) = \{[-2]^{s},[0]^{|G|-s}\}.$$ 
This is a contradiction, since $X^*(G,S)$ being a regular graph has at least one positive eigenvalue (the degree of regularity). Therefore, $\G_e^*$ and $\G_{S \cup \{e\}}^*$ 
are not isospectral. 

\sk 

\noindent $(c)$ 
In this case, we have that $\G_S^*$ and $\G_{S \cup \{e\}}^*$ are isospectral if and only if
	$$ \{[2\lambda^*_i]^{m_i}\}_{i\in I}\cup\{[0]^{|G|}\} = \{[2\lambda^*_i+1]^{m_i} \}_{i\in I}\cup \{[-1]^{|G|}\}.$$
Then, $2\lambda^*_i=-1$ and $2\lambda^*_i+1=0$, i.e.\@ $\lambda^*_i=-\tfrac{1}{2}$ for all $i \in I$. But $\lambda^*_i=-\tfrac{1}{2}$ is a contradiction since no graph can have rational non integral eigenvalues.
Another argument is that, since $X^*(G,S)$ is a regular graph, it has at least one positive eigenvalue (the degree of regularity).
Therefore, $\G_S^*$ and $\G_{S \cup \{e\}}^*$ are not isospectral. 
\end{proof}		
	
The graphs in Example \ref{exam: CayZ4} illustrates the proposition.

\begin{rem} 
A similar but more general conclusion that in Proposition \ref{prop: isospec T,T'} can be obtained by using two different subsets $S$ and $S'$ of $G$. In fact, any $MX^*(G;S,T)$ is not isospectral to $MX^*(G;S',T')$, provided that $T$ and $T'$ are of `different form'.  
More precisely, the graph $MX^*(G;S,\{e\})$ is not isospectral to 
$MX^*(G;S',S')$ and $MX^*(G;S',S'\cup\{e\})$;   
$MX^*(G;S,S)$ is not isospectral to 
$MX^*(G;S',\{e\})$ and $MX^*(G;S',S'\cup\{e\})$; and 
$MX^*(G;S,S\cup \{e\})$ is not isospectral to 
$MX^*(G;S',\{e\})$ and $MX^*(G;S',S')$ for any $S'$.
If $S'=S$, then we get Proposition \ref{prop: isospec T,T'}.
\end{rem}

\subsection{Isospectrality between $MX(G;S,T)$ and $MX^+(G;S,T)$} \label{sec: isosp}
Here we study the isospectrality between MDCGs in $\mathcal{F}$ of the form $MX(G;S,T)$ and 
$MX^+(G;S,T)$. 

We begin by showing that a mirror di-Cayley graph in $\mathcal{F}$ and its sum version 
are isospectral if and only if the corresponding underlying Cayley graphs are isospectral.

\begin{thm} \label{thm: isosp X,X+}
Let $G$ be a group, $S \subset G$ and $\mathcal{S} = \{\{e\}, S, S\cup\{e\} \}$. 
For each $T \in \mathcal{S}$, the graphs 
$MX(G;S,T)$ and $MX^{+}(G;S,T)$ are isospectral if and only if the graphs $X(G,S)$ and $X^{+}(G,S)$ are isospectral. 
\end{thm}

\begin{proof}
In the case where $T=S$ or $S\cup\{e\}$, the isospectrality statement follows directly from 
Proposition \ref{prop: spec bicayleys}. 
		
If $T=\{e\}$, put 
    $$\G_e=MX(G;S,\{e\}) \quad \text{and} \quad \G_e^+=MX^+(G;S,\{e\}).$$
By \eqref{eq: specbi}, we have that $\G_e$ 
and $\G_e^+$ are isospectral if and only if 
\begin{equation} \label{eq: Specs}
	\{[\lambda_i + 1]^{m_i},[\lambda_i-1]^{m_i}\}_{ i\in I} = \{[\lambda_j^{+} + 1]^{m_j^+},[\lambda_j^+-1]^{m_j^+}\}_{j\in J},
\end{equation}
where $Spec(X(G,S)) = \{[\lambda_i]^{m_i}\}_{i\in I}$ and $Spec(X^+(G,S)) = \{[\lambda_j^+]^{m_j^+}\}_{j \in J}$. 
		
If $X(G,S)$ and $X^+(G,S)$ are isospectral, we can assume that  $I=J$ and, after reordering if necessary, that $\lambda_i=\lambda_i^+$ and $m_i=m_i^+$ for every $i \in I$. 
Hence, 
    $$\lambda_i\pm 1=\lambda_i^+\pm 1$$ 
for any $i\in I$, and thus the graphs 
$\G_e$ and $\G_e^+$ are isospectral.

Conversely, if $\G_e$ and $\G_e^+$ are isospectral, then \eqref{eq: Specs} holds and we want to show that $X(G,S)$ and $X^+(G,S)$ are isospectral. First, notice that if $m_i \ne m_j^+$ for some $i\in I$ and $j\in J$, then 
    $$\lambda_i +1\neq \lambda_j^+\pm 1 \qquad \text{and} \qquad \lambda_i - 1\neq \lambda_j^+\pm 1.$$ 
Therefore, we can study separately the blocks of eigenvalues all having the same multiplicity. Thus, in each of such blocks, without loss of generality, we can assume that this common multiplicity is 1 for every eigenvalue in the block.
Then, consider an arbitrary block
\begin{equation} \label{eq: blocks}
\{\lambda_1+1,\lambda_1-1,\dots,\lambda_\ell+1,\lambda_\ell-1\} = \{\lambda_1^++1,\lambda_1^+-1,\dots,\lambda_\ell^++1,\lambda_\ell^+-1\}.
\end{equation}
Suppose that for each $i$ we have that 
    $$\lambda_i+1 \ne \lambda_j^++1$$ 
for all $j$. 
Then, for each $i$, the eigenvalue $\lambda_i+1$ must be the eigenvalue $\lambda_k^+ -1$ for some $k$.
Hence, we reorder the eigenvalues such that we get 
\begin{equation} \label{eq: isospec 1}
    \lambda_i+1 = \lambda_i^+-1 \qquad \quad (1\le i \le \ell),
\end{equation} 
and we consider the permutation  $\sigma \in \mathbb{S}_\ell$ such that
\begin{equation} \label{eq: isospec 2}
    \lambda_i-1 = \lambda_{\sigma(i)}^++1 \qquad \quad (1\le i \le \ell).
\end{equation}
		
Using \eqref{eq: isospec 2}, we have that
\begin{equation} \label{eq: isospec 3}	 
	\lambda_1^+=\lambda_{\sigma^{-1}(1)}-2,
\end{equation}	
and, since  $\lambda_{\sigma^{-1}(1)}=\lambda_{\sigma^{-1}(1)}^+-2$ by \eqref{eq: isospec 1}, it follows that
	$$\lambda_1^+=\lambda_{\sigma^{-1}(1)}^+-4.$$
		
On the other hand, using \eqref{eq: isospec 1} and \eqref{eq: isospec 2} repeatedly, we have that
	\begin{equation} \label{eq: isospec 4}	 
		\begin{tabular}{rcl}
			$\lambda_1^+$ &=& $\lambda_1+2=\lambda_{\sigma(1)}^++4$ \\
			&$=$& $\lambda_{\sigma(1)}+6$ = $\lambda_{\sigma^2(1)}^++8$ \\
			&$=$&  $\lambda_{\sigma^2(1)}+10$ = $\lambda_{\sigma^3(1)}^++12$ \\
			&  $\vdots$ & \\
			&$=$& $\lambda_{\sigma^{r-1}(1)}^+ + 4(r-1)+2 = \lambda_{\sigma^r(1)}^++4r = \lambda_{\sigma^{-1}(1)}^++4r$,
		\end{tabular}
	\end{equation}	
where $1 \le r \le \ell$ is such that $\sigma^r(1)=\sigma^{-1}(1)$, since $\sigma$ is a permutation.
Therefore, by putting together \eqref{eq: isospec 3} and \eqref{eq: isospec 4} we finally get 
	$$\lambda_{\sigma^{-1}(1)}^+-4 = \lambda_1^+ = \lambda_{\sigma^{-1}(1)}^++4r,$$
which is a contradiction since $r>0$. Therefore, there exists $i$ such that $\lambda_i+1=\lambda_i^++1$, that is  $\lambda_i=\lambda_i^+$.
		
Now, we do induction on the size of an arbitrary block, say $\ell$, to see that $\lambda_i=\lambda_i^+$ for all $1\le i \le \ell$. 
If $\ell=1$ we have
    $$\{\lambda_1+1,\lambda_1-1\}=\{\lambda_1^++1,\lambda_1^+-1\}$$
and by the previous reasoning we have  $\lambda_1+1=\lambda_1^++1$. 
Then, $\lambda_1-1=\lambda_1^+-1$ and hence $\lambda_1=\lambda_1^+$.
Assume by induction that if
    $$\{\lambda_1+1,\lambda_1-1,\dots,\lambda_k+1,\lambda_k-1\}=\{\lambda_1^++1,\lambda_1^+-1,\dots,\lambda_k^++1,\lambda_k^+-1\},$$
for some $1<k<\ell$, then $\lambda_i=\lambda_i^+$ for all $i=1, \ldots, k$ and suppose that
    $$\{\lambda_1+1,\lambda_1-1,\dots,\lambda_{k+1}+1,\lambda_{k+1}-1\}=\{\lambda_1^++1,\lambda_1^+-1,\dots,\lambda_{k+1}^++1,\lambda_{k+1}^+-1\}.$$
By the same argument as before, there exists $1\le i \le k+1$ such that $\lambda_i=\lambda_i^+$. 
Reordering the eigenvalues we can assume that $\lambda_{k+1}=\lambda_{k+1}^+$. 
Hence, $\lambda_{k+1}+1=\lambda_{k+1}^++1$ and $\lambda_{k+1}-1=\lambda_{k+1}^+-1$. 
So, we have that
    $$\{\lambda_1+1,\lambda_1-1,\dots,\lambda_k+1,\lambda_k-1\}=\{\lambda_1^++1,\lambda_1^+-1,\dots,\lambda_k^++1,\lambda_k^++1\}.$$
Then, by the inductive hypothesis, $\lambda_i=\lambda_i^+$ for all $i=1,\dots, k$, and hence, by finite induction, \eqref{eq: blocks} holds.
		
Therefore, we have proved that $X(G,S)$ and $X^+(G,S)$ are isospectral.
\end{proof}
	
Notice that, by Lemma \ref{lem: eigenvalues X*GS}, $X(G,S)$ and $X^+(G,S)$ can only be isospectral if $S$ is symmetric, since $Spec(X^+(G,S))$ is always real but $Spec(X(G,S))$ is real if and only if $S$ is symmetric.

\begin{exam} \label{exam: D4}
Consider the the dihedral group $\mathbb{D}_4= \langle \rho, \tau : \rho^4=\tau^2, \tau \rho \tau^{-1}=\rho^3 \rangle$, the connection set $S=\{\rho, \tau\}$, and the Cayley (sum) graphs 
	$$ \G=X(\mathbb{D}_4,S) \qquad \text{and} \qquad \G^+=X^+(\mathbb{D}_4,S).$$

They are loopless 2-regular graphs (see the Note after Proposition \ref{prop: directedness}) which are not isomorphic. In fact, $\G$ is a mixed (having both directed and undirected edges) because $S$ is neither symmetric nor antisymmetric, since $S\ne S^{-1}$ and $S \cap S^{-1} = \{id\} \ne \varnothing$. On the other hand, $S$ is antinormal because $N_{\mathbb{D}_4}(S) = \{ g\in \mathbb{D}_4 : gSg^{-1} =S\} = \{id, \rho^2\}$. Hence, $S \cap N_{\mathbb{D}_4}(S) = \varnothing$. By item ($b$) of Proposition \ref{prop: directedness} for Cayley graphs, $\G^+$ is directed. Hence, $\G \not \simeq \G^+$.

The graphs are isospectral with even spectrum $Spec(\G^*)= \{ [2]^1, [0]^6, [-2]^1\}$.
Thus, by the previous theorem, we have the isospectral pairs
	$$ \{ \G_e, \G^+_e \},  \qquad \{\G_S, \G^+_S \}, \qquad \{ \G_{S \cup e}, \G^+_{S \cup e} \}.$$
Their spectra are given by 
\begin{gather*}
	Spec(\G_e^*)= \{ [3]^1, [1]^7, [-1]^7, [-3]^1\}, \qquad Spec(\G_S^*)= \{ [4]^1, [0]^{14}, [-4]^1\}, \\[1mm]
		Spec(\G_{S \cup e}^*) = \{ [5]^1, [1]^{6}, [-1]^8, [-3]^1\}.
\end{gather*} 
Thus, $\G^*_e$ are odd, $\G^*_S$ are even, and $\G^*_{S \cup e}$ are integral, but neither even nor odd.  

Finally, it can be proved that these graphs cannot be obtained as abelian graphs (they are not isomorphic to abelian (sum) Cayley graphs).
Moreover, it is worth noticing that the pair in this example is the only pair of isospectral Cayley graphs of the form $\{X(G,S), X^+(G,S)\}$ with group $G$ of order up to 15 (see \cite{ChP}).
\hfill $\diamond$
\end{exam}

We now show that for $(G,S)$ an abelian-symmetric pair, that is $G$ abelian and $S$ symmetric not containing $0$, 
if the spectra of the Cayley graphs $X(G,S)$ and $X^+(G,S)$ are symmetric (i.e.\@ $m(\lambda^*)=m(-\lambda^*)$ for every eigenvalue $\lambda^*$ of $X^*(G,S)$), then the mirror di-Cayley (sum) graphs 
$MX(G;S,T)$ and $MX^+(G;S,T)$ are isospectral, for $T \in \mathcal{S}$ with $\mathcal{S} = \{\{0\}, S, S \cup \{0\}\}$.

\begin{coro} \label{coro: X&X+ isosp}
Let $G$ be a finite abelian group and $S$ a symmetric subset of $G$ with $0 \notin S$. For $T\in\{\{0\}, S, S\cup\{0\}\}$ we have:
		
\begin{enumerate}[$(a)$]
    \item If $Spec(X(G,S))$ and $Spec(X^+(G,S))$ are both symmetric, then $MX(G;S,T)$ and $MX^+(G;S,T)$ are isospectral.  \sk
			
    \item If $S$ is also normal and $X(G,S)$ and $X^+(G,S)$ are bipartite, then $MX(G;S,T)$ and $MX^+(G;S,T)$ are isospectral.
\end{enumerate}
\end{coro}
	
\begin{proof}
By Theorem \ref{thm: isosp X,X+}, it is enough to see that $X(G,S)$ and $X^+(G,S)$ are isospectral graphs.
For $(a)$, by Corollary $2.9$ of \cite{PV}, the graphs $X(G,S)$ and $X^+(G,S)$ are isospectral. 

For $(b)$, notice that $X(G,S)$ and $X^+(G,S)$ are undirected ($S$ being symmetric and normal, respectively). An undirected graph has symmetric spectrum if and only if it is bipartite. Again, by Corollary $2.9$ of \cite{PV}, $Spec(X(G,S))=Spec(X^+(G,S))$.
Then, in both cases, the graphs $MX(G;S,T)$ and 
$MX^+(G;S,T)$ are isospectral. 
\end{proof}

\subsection{Isospectrality between $MX^{*_1}(G_1;S_1,T_1)$ and $MX^{*_2}(G_2;S_2,T_2)$}
We now consider the more general situation; that is, isospectrality between graphs arising from different pairs $(G_1,S_1)$ and $(G_2,S_2)$. 

We will show that the mirror di-Cayley (sum) graphs $MX^{*_1}(G_1;S_1,T_1)$ and $MX^{*_2}(G_2;S_2,T_2)$, with $T_i \in  \mathcal{S}_i = \{\{e_i\}, S_i, S_i \cup \{e_i\} \}$ for $i=1,2$ of the same type (e.g., $T_1=S_1\cup \{e_1\}$ and $T_2=S_2 \cup \{e_2\}$), are isospectral if and only if the corresponding Cayley graphs $X^{*_1}(G_1,S_1)$ and $X^{*_2}(G_2,S_2)$ are isospectral. 
	 
\noindent \textit{Note}.
Here we use the notation $X^{*_1}(G_1,S_1)$ and $X^{*_2}(G_2,S_2)$ (and the same for MDCGs) to emphasize that, unless in the rest of the work where we use $X^{*}(G_1,S_1)$ and $X^{*}(G_2,S_2)$ to refer to both $X(G_1,S_1)$ and $X(G_2,S_2)$ or to $X^+(G_1,S_1)$ and $X^+(G_2,S_2)$, now we can have the crossed situations also: $X(G_1,S_1)$ and $X^+(G_2,S_2)$ or $X^+(G_1,S_1)$ and $X(G_2,S_2)$.

\begin{thm} \label{thm: gen isosp}
For $i=1,2$, let $G_i$ be a group and $S_i \subset G_i$. 
The sets of graphs  
\begin{gather*}
\{MX^{*_1}(G_1;S_1,\{e_1\}), MX^{*_2}(G_2;S_2,\{e_2\}) \}, \\ 
\{ MX^{*_1}(G_1;S_1,S_1), MX^{*_2}(G_2;S_2,S_2) \}, \\ 
\{ MX^{*_1}(G_1;S_1,S_1 \cup \{e_1\}), MX^{*_2}(G_2;S_2, S_2 \cup \{e_2\}) \},     
\end{gather*}
\nopagebreak
form isospectral pairs if and only if $X^{*_1}(G_1,S_1)$ and $X^{*_2}(G_2,S_2)$ are isospectral. 
Moreover, in the isospectral case, we must have $|G_1|=|G_2|$ and $|S_1|=|S_2|$.
\end{thm}

\begin{proof}
This proof is analogous to that of Theorem \ref{thm: isosp X,X+}, considering the graphs
$MX^{*_1}(G_1;S_1,T_1)$ and $MX^{*_2}(G_2;S_2,T_2)$ for $T_i \in\{\{e_i\},S_i,S_i\cup\{e_i\}\}$ with $i=1,2$ instead of
the graphs $MX(G;S,T)$ and $MX^{+}(G;S,T)$, for $T\in\{\{e\},S,S \cup\{e\}\}$.
Finally, we have that
$|G_1| = |G_2|$ and $|S_1| = |S_2|$,
since isospectral graphs have the same number of vertices (being the total number of eigenvalues) and the same
degree of regularity (being the principal eigenvalue).
\end{proof}

We now illustrate the previous result.
\begin{exam}
Let 
$G_1=\Z_{16}$ and $G_2=\Z_4\times\Z_4$
and consider the subsets
\begin{gather*}
    S_1=\{1,2,4,5,9,10,12,13\} \subset \Z_{16}, \\ 
    S_2=\{(0,1),(0,2),(1,0),(1,2),(2,1),(2,2),(3,1),(3,3)\} \subset \Z_4 \times \Z_4.
\end{gather*}

Let $\G_i=X(G_i,S_i)$ and $\G_i^+ = X^+(G_i,S_i)$ for $i=1,2$.
All $\G_1, \G_1^+, \G_2, \G_2^+$ are connected 8-regular graphs of 16 vertices. Since $S_1$ and $S_2$ are not symmetric, the graphs $\G_1$ and $\G_2$ are directed. 
In fact, since 
	$$S_1 \cap (-S_1) = \{4,12\} \ne \varnothing \qquad \text{and} \qquad S_2 \cap (-S_2) = \{(0,2),(2,2)\} \ne \varnothing,$$ 
the graphs are mixed (see Section 2 of \cite{PV3}).
These graphs are also loopless, since the identities $0$ and $(0,0)$ are not in the connection sets. However, $S_1$ and $S_2$ are normal, since $G_1$ and $G_2$ are abelian, and hence $\G_1^+$ and $\G_2^+$ are undirected. Notice that $\G_1^+$ has loops at vertices $1, 2, 5, 6, 9, 10, 13$ and $14$ and that $\G_2^+$ has loops at vertices $(0,1), (0,3), (1,1), (1,3), (2,1), (2,3), (3,1)$ and $(3,3)$. 
All of this implies that 
    $$\G_1 \not \simeq \G_1^+ \qquad \text{and} \qquad \G_2 \not \simeq \G_2^+.$$

Now, we study the spectra of these graphs.
Since both groups are abelian, the spectrum of the Cayley (sum) graphs can be computed using characters.

\noindent ($a$) The irreducible characters of $\Z_{16}$ are, for each $x \in \Z_{16}$, given by 
	$$ \chi_k(x) = e^{\frac{2\pi i k x}{16}}, \qquad k=0,1,\dots,15.$$
Thus, the eigenvalues of the Cayley graph $\G_1=X(G_1,S_1)$ are given by
	$$ \lambda_k = \chi_k(S_1)  = \sum_{s\in S_1} e^{\frac{2\pi i k s}{16}}$$ 
and hence 
	$$ \lambda_k = e^{\frac{\pi i k}{8}} + e^{\frac{2\pi i k}{8}} + e^{\frac{4\pi i k}{8}} + e^{\frac{5\pi i k}{8}} + e^{\frac{9\pi i k}{8}} + e^{\frac{10\pi i k s}{8}}+ e^{\frac{12 \pi i k s}{8}} + e^{\frac{13\pi i k s}{8}}, $$  
for $k=0,1,\dots,15$.
Thus, after some computations we obtain the spectrum of 
$\G_1$ which is (complex and) given by
\begin{equation} \label{eq: spec G1}
 Spec(\G_1) =  \{[8]^1,[4i]^1,[-2+2i]^2,[0]^9,[-2-2i]^2,[-4i]^1\} \subset \Z[i]. 
\end{equation}

Furthermore, by $(b)$ in Lemma \ref{lem: eigenvalues X*GS}, we have that 
\begin{equation} \label{eq: spec G1+}
	Spec(\G_1^+) =  \{ [8]^1, [4]^1, [2\sqrt 2]^2,[0]^9, [-2\sqrt 2]^2, [-4]^1\} \subset \Z[\sqrt 2]
\end{equation}
is real, and hence $Spec(\G_1) \ne Spec(\G_1^+)$.

The irreducible characters of $\Z_4\times \Z_4$ are, for each $(x,y) \in \Z_4 \times \Z_4$, given by
$$ \chi_{a,b}(x,y) = e^{\frac{2\pi i (ax+by)}{4}}, \qquad a,b\in\Z_4.$$
Thus the eigenvalues of the Cayley graph $\G_2= X(G_2,S_2)$ are given by
    $$ \lambda_{a,b} = \chi_{a,b}(S_2)= \sum_{(s,t)\in S_2}e^{\frac{2\pi i (as+bt)}{16}}, \qquad a,b\in\Z_4. $$
With a little work, after some tedious but straightforward computations, one can obtain that the spectrum is thus given by
\begin{equation} \label{eq: spec G2}
 Spec(\G_2) = \{ [8]^1,[4i]^1,[-2+2i]^2,[0]^9,[-2-2i]^2,[-4i]^1 \} \subset \Z[i]. 
\end{equation}
Again, by ($b$) in Lemma \ref{lem: eigenvalues X*GS} we get that 
\begin{equation} \label{eq: spec G2+}
 Spec(\G_2^+) =  \{ [8]^1, [4]^1, [2\sqrt 2]^2,[0]^9, [-2\sqrt 2]^2, [-4]^1\} \subset \Z[\sqrt 2].
\end{equation}
In this way, by \eqref{eq: spec G1}--\eqref{eq: spec G2+}, we have that 
$$ Spec(\G_1) = Spec(\G_2) \qquad \text{and} \qquad Spec(\G_1^+)=Spec(\G_2^+),$$
and $Spec(\G_i) \ne Spec(\G_i^+)$ for $i=1,2$, confirming the fact that $\G_i$ is not isomorphic to $\G_i^+$ for $i=1,2$.

However, the spectra of $\G_i$ and $\G_i^+$ are related by the rule $\lambda_i^+ = \rm{sign}(Im(\lambda_i)) \|\lambda_i\|$, with the convention sign$(0)=1$. Or, in other terms, the non-principal eigenvalues of $\G_i^+$ are $\pm \| \lambda_i \|$ where $\lambda_i$ are the non-principal eigenvalues of $\G_i$, for $i=1,2$.
Moreover, the graphs  $\G_i$ and $\G_i^+$ are equienergetic since they have the same energy. The energy $E(G)$ of a graph $G$ is the sum of the absolute values of its eigenvalues, and hence 
	$$ E(\G_i) = E(\G_i^+) $$  
as one can see from \eqref{eq: spec G1}--\eqref{eq: spec G2+}.

We now show that $\{\G_1,\G_2\}$ and $\{\G_1^+,\G_2^+\}$ are not pairs of isomorphic graphs. First, notice that in the Cayley graph $X(\Z_{16},S_1)$, the connection set exhibits a specific translational invariant, namely
    $$ S_1+8 = S_1 \: \pmod{16}.$$
Consequently, for any vertex $v\in\Z_{16}$, it holds that
\begin{align*}
 N(v) & = \{w\in \Z_{16}\,:\, w-v\in S_1\} \\ 
      & = \{w\in\Z_{16}\,:\,w-(v+8)\in S_1-8=S_1\}=N(v+8).
\end{align*}
This implies that every vertex $v$ in $X(G_1, S_1)$ has a twin at $v+8$.

Conversely, in the case of the Cayley graph $X(\Z_4\times\Z_4, S_2)$ one can verify that 
	$$ S_2+g \ne S_2 $$ 
for every non-zero element $g\in \Z_4\times\Z_4$.
Thus, the mapping 
    $$v \mapsto N(v)$$ 
is injective, ensuring that no vertex in this graph possesses a twin. Since the existence of twin vertices is a structural property preserved under isomorphism, the graphs $X(G_1, S_1)$ and $X(G_2, S_2)$ are not isomorphic. An analogous argument holds for the Cayley sum graphs $X^+(G_1,S_1)$ and $X^+(G_2,S_2)$. 

Next, we give the pictures of the graphs. We have marked in blue the undirected edges in the mixed graphs $\G_1$ and $\G_2$ and in red the vertices in the connection sets $S_1$ and $S_2$.

\begin{figure}[H]
\centering
\begin{tikzpicture}[
    scale=3.2,
    vertex/.style={circle,draw,inner sep=1.5pt},
    every edge/.style={line width=0.2pt}
]

\begin{scope}[xshift=-1.2cm]

\def\n{16}

\tikzset{punto/.style={circle, fill=#1, inner sep=0pt, minimum size=4pt}}

\foreach \i in {0,...,15} {
	\pgfmathparse{(\i==1)||(\i==2)||(\i==4)||(\i==5)||(\i==9)||(\i==10)||(\i==12)||(\i==13)}
	\ifnum\pgfmathresult=1
	\node[punto=red] (\i) at ({360/\n * \i}:1) {};
	\else
	\node[punto=black] (\i) at ({360/\n * \i}:1) {};
	\fi
}

\node[vertex, label=right:{0}] at (0) {};

\node[vertex, label=above:{4}] at (4) {};

\node[vertex, label=left:{8}] at (8) {};

\node[vertex, label=below:{12}] at (12) {};

\foreach \i in {0,...,15} {
    \pgfmathtruncatemacro{\j}{mod(\i+1,\n)}
    \draw[->, thin, opacity=0.6] (\i) to (\j);

    \pgfmathtruncatemacro{\j}{mod(\i+2,\n)}
    \draw[->, thin, opacity=0.6] (\i) to (\j);

    \pgfmathtruncatemacro{\j}{mod(\i+4,\n)}
    \draw[-, blue, thin, opacity=0.6] (\i) to (\j);

    \pgfmathtruncatemacro{\j}{mod(\i+5,\n)}
    \draw[->, thin, opacity=0.6] (\i) to (\j);

    \pgfmathtruncatemacro{\j}{mod(\i+9,\n)}
    \draw[->, thin, opacity=0.6] (\i) to (\j);

    \pgfmathtruncatemacro{\j}{mod(\i+10,\n)}
    \draw[->, thin, opacity=0.6] (\i) to (\j);

    \pgfmathtruncatemacro{\j}{mod(\i+12,\n)}
    \draw[-, blue, thin, opacity=0.6] (\i) to (\j);

    \pgfmathtruncatemacro{\j}{mod(\i+13,\n)}
    \draw[->, thin, opacity=0.6] (\i) to (\j);
}

\node at (0,-1.45) {$\G_1=X(\mathbb Z_{16},S_1)$};

\end{scope}

\begin{scope}[xshift=1.2cm]

\def\n{16}

\tikzset{punto/.style={circle, fill=#1, inner sep=0pt, minimum size=4pt}}

\foreach \i in {0,...,15} {
	\pgfmathparse{(\i==1)||(\i==2)||(\i==4)||(\i==5)||(\i==9)||(\i==10)||(\i==12)||(\i==13)}
	\ifnum\pgfmathresult=1
	\node[punto=red] (\i) at ({360/\n * \i}:1) {};
	\else
	\node[punto=black] (\i) at ({360/\n * \i}:1) {};
	\fi
}

\node[vertex, label=right:{0}] at (0) {};
\node[vertex, label=above:{4}] at (4) {};
\node[vertex, label=left:{8}] at (8) {};
\node[vertex, label=below:{12}] at (12) {};

\foreach \i in {0,...,15} {
    \pgfmathtruncatemacro{\jA}{mod(1-\i+\n,\n)}
    \pgfmathtruncatemacro{\jB}{mod(2-\i+\n,\n)}
    \pgfmathtruncatemacro{\jC}{mod(4-\i+\n,\n)}
    \pgfmathtruncatemacro{\jD}{mod(5-\i+\n,\n)}
    \pgfmathtruncatemacro{\jE}{mod(9-\i+\n,\n)}
    \pgfmathtruncatemacro{\jF}{mod(10-\i+\n,\n)}
    \pgfmathtruncatemacro{\jG}{mod(12-\i+\n,\n)}
    \pgfmathtruncatemacro{\jH}{mod(13-\i+\n,\n)}
    
    \draw[-, thin, opacity=0.6] (\i) to (\jA);
    \draw[-, thin, opacity=0.6] (\i) to (\jB);
    \draw[-, thin, opacity=0.6] (\i) to (\jC);
    \draw[-, thin, opacity=0.6] (\i) to (\jD);
    \draw[-, thin, opacity=0.6] (\i) to (\jE);
    \draw[-, thin, opacity=0.6] (\i) to (\jF);
    \draw[-, thin, opacity=0.6] (\i) to (\jG);
    \draw[-, thin, opacity=0.6] (\i) to (\jH);
}

\foreach \i in {1, 2, 5, 6, 9, 10, 13, 14} {
    \pgfmathsetmacro{\ang}{360/\n * \i}
    
    \draw[-, thick, opacity=0.6] (\i) .. controls +(\ang-45:0.25cm) and +(\ang+45:0.25cm) .. (\i);
}
\node at (0, -1 - 0.5) {$\G_1^+ = X^{+}(\mathbb Z_{16},S_1)$};
\end{scope}
\end{tikzpicture}
\label{fig:isospectral-cayley}
\end{figure}

\vspace{-.5cm}

\begin{figure}[H]
\centering
\begin{tikzpicture}[
    scale=2,
    vertex/.style={circle,draw,inner sep=1.5pt},
    every edge/.style={line width=0.2pt},
]

\begin{scope}[xshift=-1.9cm]

\tikzset{punto/.style={circle, fill=#1, inner sep=0pt, minimum size=4pt}}

\foreach \x in {0,...,3} {
	\foreach \y in {0,...,3} {
		\pgfmathtruncatemacro{\k}{4*\x+\y}
		\pgfmathsetmacro{\angle}{360*\k/16}
		
		\pgfmathparse{
			(\x==0 && \y==1) ||
			(\x==0 && \y==2) ||
			(\x==1 && \y==0) ||
			(\x==1 && \y==2) ||
			(\x==2 && \y==1) ||
			(\x==2 && \y==2) ||
			(\x==3 && \y==1) ||
			(\x==3 && \y==3)
		}
		
		\ifnum\pgfmathresult=1
		\node[punto=red] (\x\y) at (\angle:1.6) {};
		\else
		\node[punto=black] (\x\y) at (\angle:1.6) {};
		\fi
	}
}

\node[vertex, label=above:{$(0,0)$}] at (00) {};
\node[vertex, label=above:{$(0,2)$}] at (02) {};
\node[vertex, label=above:{$(1,0)$}] at (10) {};
\node[vertex, label=above:{$(1,2)$}] at (12) {};
\node[vertex, label=above:{$(2,0)$}] at (20) {};
\node[vertex, label=below:{$(2,2)$}] at (22) {};
\node[vertex, label=below:{$(3,0)$}] at (30) {};
\node[vertex, label=below:{$(3,2)$}] at (32) {};

\foreach \x in {0,...,3} {
	\foreach \y in {0,...,3} {
		
		\foreach \a/\b in {0/2,2/2} {
			
			\pgfmathtruncatemacro{\xx}{mod(\x+\a,4)}
			\pgfmathtruncatemacro{\yy}{mod(\y+\b,4)}
			
			\ifnum\y<\yy
			\draw[blue, thin, opacity=0.6] (\x\y) -- (\xx\yy);
			\fi
			\ifnum\x<\xx
			\draw[blue, thin, opacity=0.6] (\x\y) -- (\xx\yy);
			\fi
			
		}
	}
}

\foreach \x in {0,...,3} {
	\foreach \y in {0,...,3} {
		\foreach \a/\b in {0/1,1/0,1/2,2/1,3/1,3/3} {
			
			\pgfmathtruncatemacro{\xx}{mod(\x+\a,4)}
			\pgfmathtruncatemacro{\yy}{mod(\y+\b,4)}
			
			\draw[->, thin, opacity=0.6] (\x\y) -- (\xx\yy);
		}
	}
}

\node at (0,-2.2) {$\G_2=X(\mathbb Z_4\times\mathbb Z_4,S_2)$};

\end{scope}

\begin{scope}[xshift=1.9cm]

\tikzset{punto/.style={circle, fill=#1, inner sep=0pt, minimum size=4pt}}

\foreach \x in {0,...,3} {
	\foreach \y in {0,...,3} {
		\pgfmathtruncatemacro{\k}{4*\x+\y}
		\pgfmathsetmacro{\angle}{360*\k/16}
		
		\pgfmathparse{
			(\x==0 && \y==1) ||
			(\x==0 && \y==2) ||
			(\x==1 && \y==0) ||
			(\x==1 && \y==2) ||
			(\x==2 && \y==1) ||
			(\x==2 && \y==2) ||
			(\x==3 && \y==1) ||
			(\x==3 && \y==3)
		}
		
		\ifnum\pgfmathresult=1
		\node[punto=red] (\x\y) at (\angle:1.6) {};
		\else
		\node[punto=black] (\x\y) at (\angle:1.6) {};
		\fi
	}
}

\node[vertex, label=above:{$(0,0)$}] at (00) {};
\node[vertex, label=above:{$(0,2)$}] at (02) {};
\node[vertex, label=above:{$(1,0)$}] at (10) {};
\node[vertex, label=above:{$(1,2)$}] at (12) {};
\node[vertex, label=above:{$(2,0)$}] at (20) {};
\node[vertex, label=below:{$(2,2)$}] at (22) {};
\node[vertex, label=below:{$(3,0)$}] at (30) {};
\node[vertex, label=below:{$(3,2)$}] at (32) {};

\foreach \x in {0,...,3} {
    \foreach \y in {0,...,3} {
        \foreach \a/\b in {0/1,0/2,1/0,1/2,2/1,2/2,3/1,3/3} {
            \pgfmathtruncatemacro{\xx}{mod(\a-\x+4,4)}
            \pgfmathtruncatemacro{\yy}{mod(\b-\y+4,4)}
            \draw[-, thin, opacity=0.6] (\x\y) -- (\xx\yy);
        }
    }
}

\node at (0,-2.2) {$\G_2^+=X^{+}(\mathbb Z_4\times\mathbb Z_4,S_2)$};

\foreach \x/\y in {0/1,0/3,1/1,1/3,2/1,2/3,3/1,3/3} {
    \pgfmathtruncatemacro{\k}{4*\x+\y}
    \pgfmathsetmacro{\angle}{360*\k/16}

    \draw[-, thick, opacity=0.6, looseness=12]
      (\x\y)
      to[out=\angle+45, in=\angle-45]
      (\x\y);
}

\end{scope}
\end{tikzpicture}
\label{fig:sum-cayley}
\end{figure}

Then, $\{ \G_1, \G_2\}$ and $\{\G_1^+, \G_2^+\}$ 
are two pairs of isospectral non-isomorphic graphs. 
Notice that we have $|G_1|=|G_2|=16$ and $|S_1|=|S_2|=8$.

\noindent $(b)$
Now, by Theorem \ref{thm: gen isosp}, we see that 
\begin{gather*}
\{MX(G_1;S_1,\{e\}), MX(G_2;S_2,\{e\}) \}, \\ 
\{ MX(G_1;S_1,S_1), MX(G_2;S_2,S_2) \}, \\ 
\{ MX(G_1;S_1,S_1 \cup \{e\}), MX(G_2;S_2, S_2 \cup \{e\}) \},    
\end{gather*} 
are three isospectral and non-isomorphic pairs of Cayley graphs, and
\begin{gather*}
\{MX^+(G_1;S_1,\{e\}), MX^+(G_2;S_2,\{e\}) \}, \\ 
\{MX^+(G_1;S_1,S_1), MX^+(G_2;S_2,S_2) \}, \\ 
\{MX^+(G_1;S_1,S_1 \cup \{e\}), MX^+(G_2;S_2, S_2 \cup \{e\}) \}, 
\end{gather*} 
are three isospectral and non-isomorphic pairs of Cayley sum graphs.

By Remark~\ref{rem: ring of integers}, the first three pairs of graphs have complex spectrum contained in $\Z[i]$ (Gaussian integers)
and the second three pairs of graphs have real spectrum contained in $\Z[\sqrt 2]$. So the graphs are not integral, but have the spectra contained in some ring of integers of quadratic fields.
\hfill $\diamond$ 
\end{exam}

\begin{rem}
In \cite{ChP}, we study Cayley graphs $X(G,S)$ and Cayley sum graphs $X^+(G,S)$, for groups $G$ with $|G|\le 15$. 
In particular, we give all non-isomorphic isospectral pairs of the form $\{ X^{*_1}(G_1,S_1), X^{*_2}(G_2,S_2)\}$. 
Thus, combining the results and tables in \cite{ChP} for Cayley (sum) graphs with Theorems~\ref{thm: isosp X,X+} and \ref{thm: gen isosp}, one can obtain all isospectral pairs of mirror di-Cayley graphs of the form 
	$$ \{MX^{*_1}(G_1,S_1,T_1), MX^{*_2}(G_2,S_2,T_2)\}, \qquad |G_1|, |G_2| \le 15.$$ 
In particular, we can obtain all such isospectral pairs with even and odd spectrum.
\end{rem}

\section{Even and odd isospectral Cayley graphs} \label{sec: even/odd isosp}
We now provide a nice source of examples of isospectral pairs of integral MDCGs of the form $MX(G;S,T)$ and $MX^+(G;S,T)$ with $T\in \mathcal{S}$. To this end, we consider as the group $G$ any finite commutative ring $R$ with identity.

Let $R$ be a finite commutative ring with identity. Then, it is well-known that $R$ has an Artin's decomposition (unique up to isomorphism and the order of the factors)
\begin{equation} \label{eq: Artin}
    R = R_1 \times \cdots \times R_s,    
\end{equation}  
where $R_i$ is a local ring with maximal ideal $\frak{m}_i$ for $i=1,\ldots, s$. Put               
    $$ r_i=|R_i| \qquad \text{and} \qquad                             m_i=|\frak{m}_i|$$ 
for any $i=1,\ldots, s$.
A similar decomposition as in \eqref{eq: Artin} holds for the group of units 
\begin{equation} \label{eq: Artin R*}
R^* = R_1^* \times \cdots \times R_s^*.    
\end{equation} 

Consider the graphs
    $$ G_R=X(R,R^*) \qquad \text{and} \qquad G_R^+ = X^+(R,R^*). $$ 
The decompositions \eqref{eq: Artin} and \eqref{eq: Artin R*}  imply that 
\begin{equation} \label{eq: GR kronecker}
G_R = G_{R_1} \otimes \cdots \otimes G_{R_s} \qquad \text{and} \qquad G_R^+ = G_{R_1}^+ \otimes \cdots \otimes G_{R_s}^+,
\end{equation} 
with $G_{R_i} = X(R_i,R_i^*)$, $G_{R_i}^+ = X^+(R_i,R_i^*)$ for $i=1,\ldots, s$, and where $\otimes$ denotes the Kronecker product (also direct product $\times$).

Next, we show that if $R$ is a finite commutative ring as in \eqref{eq: Artin} with at least one factor of even size and one of odd size, then the graphs          
\begin{equation} \label{eq: GRTs}   
    G_{R,T}:= MX(R;R^*,T) \qquad \text{and} \qquad G_{R,T}^+ :=MX^+(R;R^*,T),
\end{equation} 
with $T \in \mathcal{R} =\{ \{0\}, R^*, R^*\cup \{0\} \}$, are isospectral.
	
\begin{prop} \label{prop: isosp R}
Let $R$ be a finite commutative ring with identity with Artin's decomposition as in \eqref{eq: Artin}. 
Suppose that at least one $R_i$ has even size of the form $r_i= 2m_i$ 
and at least one $R_j$ is of odd size. Then, the mirror di-Cayley (sum) graphs 
    $$ \{ G_{R,T}, \;  G_{R,T}^+ \},$$ 
are integral isospectral (and generically non-isomporphic) for every choice of $T \in \mathcal{R}$, where $\mathcal{R} =\{ \{0\}, R^*, R^*\cup \{0\} \}$. 
Moreover, $\{ G_{R,R^*}, G_{R,R^*}^+ \}$ and $\{ G_{R,R^* \cup \{0\}}, G_{R,R^* \cup \{0\}}^+ \}$ are two pairs of even and odd graphs, respectively. 
\end{prop}

\begin{proof}
Consider the unitary Cayley (sum) graphs associated to the ring $R$, that is $G_R=X(R,R^*)$ and $G_R^+ = X^+(R,R^*)$. 
By Lemma 3.4 of \cite{PV}, $G_R$ and $G_R^+$ are different (bipartite) isospectral graphs, which are generically non-isomorphic. 

Now, by Theorem~\ref{thm: isosp X,X+} or Corollary~\ref{coro: X&X+ isosp}, we have that each of the three pairs of mirror di-Cayley (sum) graphs 
    $\{G_{R,T}, \,  G_{R,T}^+\}$ 
corresponding to $T=\{0\}$, $T=R^*$ or $T=R^*\cup \{0\}$ are isospectral (and generically non-isomorphic).

By Proposition \ref{prop: GRR integral}, we know that the spectrum of $G_{R,T}^*$ are integral for any $T \in \mathcal{R}$, and that $G^*_{R,R^*}$ is even and $G^*_{R,R^* \cup \{0\}}$ is odd.
\end{proof}

Examples of non-isomorphic local rings of size $p^n$ are 
\begin{equation}  \label{eq: some local rings}
    \Z_{p^n}, \qquad \mathbb{F}_{p^m}[x]/(x^t) \qquad \text{and} \qquad GR(p^m, t)
\end{equation} 
where $p$ is a prime and $n=mt$.

\begin{coro} \label{coro: R1 R1'}
Let $R$ be a finite commutative ring with identity with Artin's decomposition $R=R_1 \times R_2 \times R_3 \times \cdots \times R_s$, where $R_1$ is some of the rings in \eqref{eq: some local rings} for $p=2$, $R_2$ is some of the rings in \eqref{eq: some local rings} for $p$ odd, and 
$R_3, \ldots, R_s$ are arbitrary local rings.
Then, for each $T \in \{ \{0\}, R^*, R^* \cup \{0\}\}$, we have the pair 
    $$ \{G_{R, T}, \; G^+_{R, T} \} $$
of integral isospectral (generically non-isomporphic) MDCGs.  
Moreover, $G_{R,R^*}$ and $G_{R,R^*}^+$ are even graphs and $G_{R,R^* \cup \{0\}}$ and $G_{R,R^* \cup \{0\}}^+$ are odd graphs.
\end{coro}

\begin{proof} 
Automatic by Proposition \ref{prop: isosp R}.
\end{proof}

Next, we give the minimal examples provided by the previous corollary, also illustrating Corollaries \ref{coro: integral} and \ref{coro: symmetric}.

\begin{exam}  \label{exam: Z4Z3}
The smallest ring that we can take in Corollary \ref{coro: R1 R1'} is 
   $$ R=\Z_4 \times \Z_3$$ 
(or $R'= \Z_2[x]/(x^2) \times \Z_3$, or $R''= GR(2^2,2) \times \Z_3$) of order 12 (which is obviously $\Z_{12}$). 
For each choice of $T \in \mathcal{R}$, with $\mathcal{R} = \{\{0\}, R^*, R^*\cup \{0\}\}$, we have the pair 
$$\{ MX(\Z_4 \times  \Z_3, \Z_4^* \times \Z_3^*, T), MX^+(\Z_4 \times  \Z_3, \Z_4^* \times \Z_3^*, T) \} $$
of $8$-regular isospectral graphs of order 24. 

Notice that $G_{\Z_4}=G_{\Z_4}^+$ 
(see the comment previous to Corollary~\ref{coro: SpecGRR*}), but $G_{\Z_3} \not \simeq G_{\Z_3}^+$ since, by Proposition 3.2 in \cite{PV}, the graphs $G_{\Z_3}$ and $G_{\Z_3}^+$ are not-isospectral.
Thus, by \eqref{eq: GR kronecker}, we have that 
\begin{equation} \label{eq: G4G3}
    G_R = G_{\Z_4} \otimes G_{\Z_3} \qquad \text{and} \qquad G_R^+ = G_{\Z_4}^+ \otimes G_{\Z_3}^+
\end{equation} 
and hence $G_R \not \simeq G_R^+$. 
This implies that the MDCGs are also non-isomorphic, i.e.\@ 
    $$MX(R; R^*, T) \not \simeq MX^+(R; R^*, T)$$ 
for every $T \in \mathcal{R}$. 

We now check the isospectrality result by giving the spectra of all these graphs explicitly.
By \eqref{eq: Spec GR local} and \eqref{eq: Spec GR+ local}, with $r=4$ and $m=2$, we have that 
\begin{gather*}
Spec(G_{\Z_4}) = Spec(G_{\Z_4}^+) = \{ [2]^1, [0]^2, [-2]^1 \}, \\
Spec(G_{\Z_3}) = \{ [2]^1, [-1]^2 \} \qquad \text{and} \qquad Spec(G_{\Z_3}^+) = \{ [2]^1, [1]^1, [-1]^1 \}.   
\end{gather*}

Now, to compute the spectrum of $G_R$ and $G_R^+$ we use \eqref{eq: G4G3}, the above computations and \eqref{spec}, obtaining
\begin{align*}
    Spec(G_R) & = \{ [4]^1, [0]^2, [-4]^1, [-2]^2, [0]^4, [2]^2 \},\\
    Spec(G_R^+) & = \{ [4]^1, [2]^1, [-2]^1, [0]^6, [-4]^1, [-2]^1, [2]^1 \}, 
\end{align*}
and thus 
	$$ Spec(G_R) = \{ [4]^1, [2]^2, [0]^6, [-2]^2, [-4]^1 \} = Spec(G_R^+) .$$

From this and Proposition \ref{prop: spec bicayleys} we get the spectra of all the mirror di-Cayley (sum) graphs associated to $G_R$; namely,
\begin{align*}
   Spec(G_{R,0}^*) &= \{ [5]^1, [3]^3, [1]^8, [-1]^8, [-3]^3, [-5]^1 \},\\
   Spec(G_{R,R^*}^*) &= \{ [8]^1, [4]^2, [0]^{18}, [-4]^2, [-8]^1 \},
   \\
   Spec(G_{R,R^*\cup \{0\}}^*) &= \{ [9]^1, [5]^2, [1]^6, [-3]^2, [-7]^1, [-1]^{12}\}. 
\end{align*}
From this, we can see that $G_{R,R^*}^*$ has even and symmetric spectrum (hence bipartite) and $G_{R,R^* \cup \{0\}}^*$ has odd and non-symmetric spectrum (hence non-bipartite). The graph $G_{R,0}^*$ is odd, because in this case $G_{R}^*$ is even. 
\hfill $\diamond$ 
\end{exam}

We are finally in a position to answer affirmatively the question posed in the introduction. 
We show that there exist pairs of isospectral Cayley (sum) graphs of the form $\{ X(G,S), X^+(G,S)\}$ that both have even spectrum or both have odd spectrum for a particular choice of the pair $(G,S)$.

\begin{thm} \label{thm: main}
There are integral isospectral pairs of unitary Cayley (sum) graphs $\{ X(R,R^*), X^+(R,R^*)\}$, with $R$ a finite commutative ring with identity, such that:

\noindent $(a)$ they are both bipartite with even symmetric spectrum; or

\noindent $(b)$ they are both non-bipartite with odd non-symmetric spectrum.
\end{thm}

\begin{proof}
Let $R$ be a finite commutative ring with identity with Artin's decomposition 
    $$R=R_1 \times R_2 \times R_3 \times \cdots \times R_s,$$ 
where $R_1$ is (for instance) some of the rings in \eqref{eq: some local rings} for $p=2$, $R_2$ is (for instance) some of the rings in \eqref{eq: some local rings} for $p$ odd, and 
$R_3, \ldots, R_s$ are arbitrary local rings.
Then, by Corollary \ref{coro: R1 R1'}, for each $T$ in the set $\{ \{0\}, R^*, R^* \cup \{0\}\}$, we have the pair $ \{ G_{R, T}, \; G^+_{R, T} \} $
of integral isospectral MDCGs (generically non-isomorphic).   
Moreover, $G_{R,R^*}$ and $G_{R,R^*}^+$ are even graphs and $G_{R,R^* \cup \{0\}}$ and $G_{R,R^* \cup \{0\}}^+$ are odd graphs.

Also, by Corollary \ref{coro: symmetric}, we have that $G_{R,R^*}$ and $G_{R,R^*}^+$ have symmetric spectrum and, hence, are bipartite graphs; and that $G_{R,R^* \cup \{0\}}$ and $G_{R,R^* \cup \{0\}}^+$ have non-symmetric spectrum and, hence, are non-bipartite graphs. 

Finally, by Proposition \ref{prop: cayley structure}, the mirror di-Cayley (sum) graphs $G_{R, T}$ and $G^+_{R, T}$ are Cayley graphs over the abelian group $R \times \Z_2$. 
\end{proof}

Our method of constructing MDCGs pairs of isospectral graphs with even or odd spectrum (as in Proposition \ref{prop: isosp R} or Theorem \ref{thm: main}), uses a finite commutative ring $R$
with at least one local factor of even characteristic (size) and one of odd characteristic (size). 
As Cayley graphs, they are defined over the group $G=R \times \Z_2$, which is also a ring in Artin's form. Hence, this process can be iterated to obtain infinite sequences of such isospectral pairs.

\begin{coro}
Let $R$ a finite commutative ring with identity and let $\{G_{R,S}, G^+_{R,S} \}$ be an isospectral pair of Cayley (sum) graphs as obtained in Proposition \ref{prop: isosp R} or in Corollary~\ref{coro: R1 R1'}. 
For each $n \in \N$, we have the pair  
    $$\{ \G_{n} =  X(G \times \Z_2^n, S \times \Z_2^n ), \quad \G_n^+ = X^+(G \times \Z_2^n, S \times \Z_2^n) \}, $$ 
of isospectral Cayley (sum) graphs with integral even spectrum.
\end{coro}

\begin{proof}
This is a consequence of Proposition \ref{prop: isosp R} or Corollary \ref{coro: R1 R1'} and the repetitive use of Proposition \ref{prop: cayley structure}. 
\end{proof}

\begin{exam}
Consider the ring $\Z_4\times \Z_3$. By Example \ref{exam: Z4Z3} and the previous corollary, the pair of graphs $\{G_{R,R^*}(n), G^+_{R,R^*}(n) \}$, where
$$ G_{R,R^*}(n) = MX \big( \Z_4 \times  \Z_3 \times \Z_2^n, \Z_4^* \times \Z_3^* \times (\Z_2^*)^n,\Z_4^* \times \Z_3^* \times (\Z_2^*)^n \big),$$
are integral isospectral with even spectrum, for every natural number $n$.
\hfill $\diamond$ 
\end{exam}

\subsection*{Acknowledgements}
The second author wishes to thank the Guangdong Technion Israel Institute of Technology (GTIIT) in Shantou, China, for its hospitality, facilities, and friendly atmosphere, during his academic visit on winter 2025, where this work was almost finished.

\end{document}